\documentclass[runningheads,11points]{llncs}

\usepackage{latexsym,amssymb,enumerate,float,amsmath}
\usepackage{graphicx,hyperref,cite,color}
\usepackage[english]{babel}
\usepackage[utf8x]{inputenc}

\usepackage{vmargin}
\setmarginsrb{4.37cm}{1.4cm}{4.37cm}{3.2cm}{1.4cm}{1.6cm}{0.8cm}{1.4cm}

\renewenvironment{proof}[1][]{\par \noindent {\bf Proof:#1}\ }{\hfill$\Box$\medskip}
\newenvironment{proofof}[1][]{\par \noindent {\bf Proof of Theorem~\ref{th:approxDelta3}:#1}\ }{\hfill$\Box$\medskip}

\newcommand{\drain}{\vdash}
\newcommand{\forg}{\bot}
\newcommand{\select}{\top}

\newcommand{\sm}{\setminus}
\newcommand{\gm}{-}

\usepackage{stmaryrd}

\usepackage{tikz}

\tikzstyle{vertex}=[
  circle,
  draw,
  thick,
  inner sep=0.1cm,
  minimum size=2mm
]

\tikzstyle{select}=[
line width=2,
red,
]

\usetikzlibrary{shapes,snakes}

\newcommand{\pa}[2]{\node[vertex,circle,fill=yellow] (#2) at (\x+#1+\y) {};}
\newcommand{\paa}[3]{\node[vertex,circle,fill=yellow,#3] (#2) at (\x+#1+\y) {};}

\newcommand{\pb}[2]{\node[vertex, rectangle,fill=green] (#2) at (\x+#1+\y) {};}
\newcommand{\pbb}[3]{\node[vertex, rectangle,fill=green,#3] (#2) at (\x+#1+\y) {};}
\newcommand{\pc}[2]{\node[vertex,rectangle, fill=green] (c) at (\x+#1+\y) {#2};}

\newcommand{\outh}[1]{\draw (#1) -- +(0.3, 0.1);}
\newcommand{\outb}[1]{\draw (#1) -- +(0.3, -0.1);}

\newcommand{\capt}[2]{\draw[white] (#1) -- node[above,sloped,black] {#2} +(2,0)}

\newcommand{\outtt}[1]{\draw (#1) -- +(-0.2, 0.6);}

\newcommand{\com}[2]{\node at (\x+#1+\y) {#2};}

\title{Uniquely restricted matchings\\and edge colorings}

\author{Julien Baste\inst{1}, Dieter Rautenbach\inst{2}, and Ignasi Sau\inst{1,3}}

\authorrunning{Julien Baste, Dieter Rautenbach, and  Ignasi Sau}
\titlerunning{Uniquely restricted matchings and edge colorings}

\institute{CNRS, LIRMM, Universit\'{e} de Montpellier, Montpellier, France\\
\email{baste@lirmm.fr,  sau@lirmm.fr}
\and
Institute of Optimization and Operations Research, Ulm University, Ulm, Germany\\ \email{dieter.rautenbach@uni-ulm.de}
\and
Departamento de Matem\'{a}tica, Universidade Federal do Cear\'{a}, Fortaleza, Brazil}

\begin{document}
\maketitle

\begin{abstract}
A matching in a graph is {\it uniquely restricted} if no other matching covers exactly the same set of vertices.
This notion was defined by Golumbic, Hirst, and Lewenstein and studied in a number of articles.
Our contribution is twofold.
We provide approximation algorithms
for computing a uniquely restricted matching of maximum size
in some bipartite graphs.
In particular, we achieve a ratio of $9/5$ for subcubic bipartite graphs,
improving over a $2$-approximation algorithm proposed by Mishra.
Furthermore, we study the {\it uniquely restricted chromatic index} of a graph,
defined as the minimum number of uniquely restricted matchings into which its edge set can be partitioned.
We provide tight upper bounds in terms of the maximum degree
and characterize all extremal graphs.
Our constructive proofs yield efficient algorithms to determine the corresponding edge colorings.\\[3mm]
\end{abstract}

\section{Introduction}\label{sec:intro}

Matchings in graphs are among the most fundamental and well-studied objects in combinatorial optimization~\cite{lopl,yuli}. While classical matchings lead to many efficiently solvable problems, more restricted types of matchings~\cite{stva} are often intractable;
induced matchings~\cite{brmo,ca1,ca2,fagyshtu,fashgytu,hera,hoqitr,jorasa,kamnmu,loz}
being a prominent example.
Here we study the so-called uniquely restricted matchings,
which were introduced by Golumbic, Hirst, and Lewenstein~\cite{gohile}
and studied in a number of papers~\cite{lema,lm2,lm3,mi,frjaja,peraso}.
We also consider the corresponding edge coloring notion.

Before we explain our contribution and discuss related research,
we collect some terminology and notation
(cf. e.g.~\cite{diestel2005} for undefined terms).
We consider finite, simple, and undirected graphs.
A {\it matching} in a graph $G$~\cite{lopl} is a set of pairwise non-adjacent edges of $G$.
For a matching $M$, let $V(M)$ be the set of vertices incident with an edge in $M$.
A matching $M$ in a graph $G$ is {\it induced}~\cite{fagyshtu} if the subgraph
$G[V(M)]$ of $G$ induced by $V(M)$ is $1$-regular.
Golumbic, Hirst, and Lewenstein~\cite{gohile} define a matching $M$ in a graph $G$ to be {\it uniquely restricted}
if there is no matching $M'$ in $G$ with $M'\not=M$ and $V(M')=V(M)$,
that is, no other matching covers exactly the same set of vertices.
It is easy to see that a matching $M$ in $G$ is uniquely restricted
if and only if there is no {\it $M$-alternating cycle} in $G$,
which is a cycle in $G$ that alternates between edges in $M$ and edges not in $M$.
Let
the {\it matching number} $\nu(G)$,
the {\it strong matching number} $\nu_s(G)$, and
the {\it uniquely restricted matching number} $\nu_{ur}(G)$ of $G$
be the maximum size of a matching,
an induced matching, and
a uniquely restricted matching in $G$, respectively.
Since every induced matching is uniquely restricted, we obtain
$$\nu_s(G)\leq \nu_{ur}(G)\leq \nu(G)$$
for every graph $G$.

Each type of matching naturally leads to an edge coloring notion.
For a graph $G$, let $\chi'(G)$ be the {\it chromatic index} of $G$,
which is the minimum number of matchings
into which the edge set $E(G)$ of $G$ can be partitioned.
Similarly, let the {\it strong chromatic index} $\chi'_s(G)$~\cite{fashgytu}
and the {\it uniquely restricted chromatic index} $\chi'_{ur}(G)$ of $G$
be the minimum number of induced matchings and uniquely restricted matchings
into which the edge set of $G$ can be partitioned, respectively.
A partition of the edges of a graph $G$ into uniquely restricted matchings
is a \emph{uniquely restricted edge coloring} of $G$.
Another related notion are {\it acyclic edge colorings},
which are partitions of the edge set into matchings
such that the union of every two of the matchings is a forest.
The minimum number of matchings in an acyclic edge coloring of a graph $G$
is its {\it acyclic chromatic index} $a'(G)$~\cite{alsuza,fi}.
Exploiting the obvious relations between the different edge coloring notions,
we obtain
\begin{eqnarray}\label{echichain}
\chi'(G)\leq a'(G)\leq \chi'_{ur}(G)\leq \chi'_s(G)
\end{eqnarray}
for every graph $G$.

%For an integer $k$, let $[k]$ denote the set of positive integers between $1$ and $k$.
%\jul{Strange formulation.}

\medskip

\noindent Stockmeyer and Vazirani~\cite{stva}
showed that computing the strong matching number is \texttt{NP}-hard.
Their result was strengthened in many ways,
and also restricted graph classes where the strong matching number can be determined efficiently were studied~\cite{brmo,ca1,ca2,loz}.
Golumbic, Hirst, and Lewenstein~\cite{gohile} showed that
it is {\texttt {NP}}-hard to determine $\nu_{ur}(G)$
for a given bipartite or split graph $G$.
Mishra~\cite{mi} strengthened this
by showing that it is not possible to approximate $\nu_{ur}(G)$
within a factor of $O(n^{\frac{1}{3}- \epsilon})$
for any $\epsilon>0$, unless \texttt{NP}$=$\texttt{ZPP},
even when restricted to bipartite, split, chordal or comparability graphs of order $n$.
Furthermore, he showed that $\nu_{ur}(G)$ is {\texttt {APX}}-complete
for subcubic bipartite graphs.

On the positive side,
Golumbic, Hirst, and Lewenstein~\cite{gohile}
described efficient algorithms that determine $\nu_{ur}(G)$
for cacti, threshold graphs, and proper interval graphs.
Solving a problem from~\cite{gohile},
Francis, Jacob, and Jana~\cite{frjaja} described an efficient algorithm for $\nu_{ur}(G)$ in interval graphs.
Solving yet another problem from~\cite{gohile},
Penso, Rautenbach, and Souza~\cite{peraso} showed that the graphs $G$
with $\nu(G)=\nu_{ur}(G)$ can be recognized in polynomial time.
Complementing his hardness results,
Mishra~\cite{mi} proposed a $2$-approximation algorithm for cubic bipartite graphs.

While $\chi'(G)$ of a graph $G$ of maximum degree $\Delta$ is either $\Delta$ or $\Delta+1$~\cite{vi},
Erd\H{o}s and Ne\v{s}et\v{r}il~\cite{fashgytu} conjectured $\chi_s'(G)\leq \frac{5}{4}\Delta^2$, and much of the research on the strong chromatic index is motivated by this conjecture.
Building on earlier work of Molloy and Reed~\cite{more},
Bruhn and Joos~\cite{brjo} showed $\chi_s'(G)\leq 1.93\Delta^2$ provided that $\Delta$ is sufficiently large.
For further results on the strong chromatic index we refer to~\cite{an,hov,hmrv,fashgytu}.

Fiam\v{c}ik~\cite{fi} and Alon, Sudakov, and Zaks~\cite{alsuza} conjectured that
every graph of maximum degree $\Delta$ has an acyclic edge coloring using no more than $\Delta+2$ colors.
See~\cite{espa,caperewa} for further references and the currently best known results concerning general graphs and graphs of large girth.

In view of the famous open conjectures on $\chi_s'(G)$ and $a'(G)$,
the inequality chain (\ref{echichain})
motivates to study upper bounds on $\chi_{ur}(G)$
in terms of the maximum degree $\Delta$ of a graph $G$.

\medskip

\noindent Our contribution is twofold. We present approximations algorithms for  $\nu_{ur}(G)$ in some bipartite graphs in Sections~\ref{sec:approx}-\ref{ap:Delta3},
and tight bounds on $\chi_{ur}'(G)$ in Section \ref{sec:edge-colorings}.

Concerning the algorithms, improving on Mishra's $2$-approximation algorithm~\cite{mi},
we describe in Section~\ref{ap:Delta3} a $9/5$-approximation algorithm for computing $\nu_{ur}(G)$ of a given bipartite subcubic graph $G$.
This algorithm requires some complicated preprocessing based on
detailed local analysis. In order to illustrate our general approach in a cleaner setting,
we first describe in Section~\ref{sec:approx} algorithms for $C_4$-free bipartite graphs of arbitrary maximum degree.

Concerning the uniquely restricted chromatic index,
we achieve best-possible upper bounds in terms of the maximum degree,
and even characterize all extremal graphs.
Since our proofs are constructive,
it is easy to extract efficient algorithms finding the corresponding edge colorings.

We conclude with some open problems in Section \ref{sec:conclusions}.

\section{Approximation algorithms for $C_4$-free bipartite graphs}\label{sec:approx}

Before we proceed to present the $9/5$-approximation algorithm for subcubic bipartite graphs in Section~\ref{ap:Delta3}, we first describe in this section an approximation algorithm for the $C_4$-free case.
The proof of the next lemma contains the main algorithmic ingredients.
Note that the size of the smaller partite set in a bipartite graph
is always an upper bound on the uniquely restricted matching number.

For an integer $k$, let $[k]$ denote the set of positive integers between $1$ and $k$. For a graph $G$, let $n(G)$ denote its number of vertices.
%\jul{Strange formulation.}

\begin{lemma}\label{lemmanew1}
Let $\Delta\geq 3$ be an integer.
If $G$ is a connected $C_4$-free bipartite graph of maximum degree at most $\Delta$
with partite sets $A$ and $B$
such that every vertex in $A$ has degree at least $2$,
and some vertex in $B$ has degree less than $\Delta$,
then $G$ has a uniquely restricted matching $M$
of size at least
$\frac{(\Delta-1)^2+(\Delta-2)}{(\Delta-1)^3+(\Delta-2)}|A|.$
Furthermore, such a matching can be found in polynomial time.
\end{lemma}
\begin{proof}
We give an algorithmic proof of the lower bound such that the running time of the corresponding algorithm is polynomial in $n(G)$, which immediately implies the second part of the statement.
Therefore, let $G$ be as in the statement.
Throughout the execution of our algorithm, as illustrated in Fig.~\ref{fig:example}, we maintain a pair $(U,M)$ such that
\begin{enumerate}[(a)]
\item\label{p0} $U$ is a subset of $V(G)$,
\item\label{p1} $M$ is a uniquely restricted matching with $V(M)\subseteq U$,
\item\label{p2} no vertex in $B\cap U$ has a neighbor in $A\setminus U$,
\item\label{p3} no vertex in $B\setminus U$ has all its neighbors in $A\cap U$,
\item\label{p4} if

$s$ vertices in $A\cap U$ are incident with an edge in $M$,

$d$ vertices in $A\cap U$ are not incident with an edge in $M$ but have a neighbor in $B\setminus U$, and

$f$ vertices in $A\cap U$ are neither incident with an edge in $M$ nor have a neighbor in $B\setminus U$, then
\begin{eqnarray}\label{e1}
(\Delta-1)^2\Big((\Delta-2)s-(d+f)\Big)\geq (\Delta-2)f.
\end{eqnarray}
\end{enumerate}

\begin{figure}[htp]
\vspace{-.65cm}
\begin{center}
\includegraphics[width=.56\textwidth]{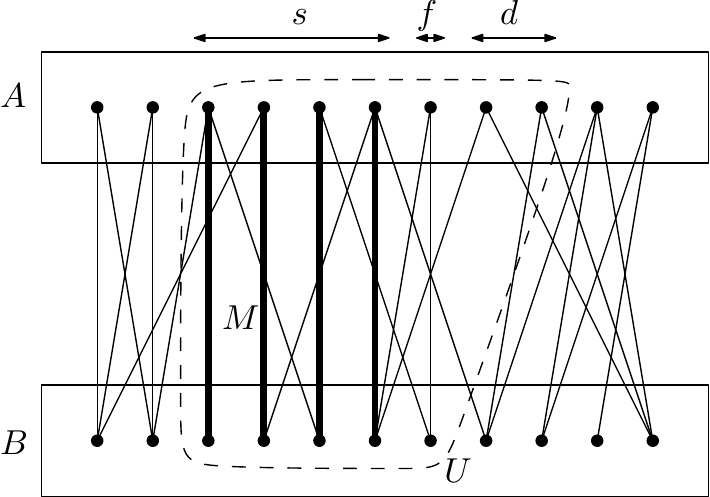}
\end{center}\vspace{-.25cm}
\caption{Example for $\Delta = 3$ of the parameters defined in the proof of Lemma~\ref{lemmanew1}. The set $U$ is dashed, and the uniquely restricted matching $M$ corresponds to the thicker edges.}
\label{fig:example}
\end{figure}
\noindent Initially, let $U$ and $M$ be empty sets.
Note that properties (\ref{p0}) to (\ref{p4}) hold.

As long as $U$ is a proper subset of $V(G)$,
we iteratively replace the pair $(U,M)$ with a pair $(U',M')$ such that
$U$ is a proper subset of $U'$,
$M$ is a proper subset of $M'$, and
properties (\ref{p0}) to (\ref{p4}) hold for $(U',M')$.
Let $s'$, $d'$, and $f'$ denote the updated values considered in (\ref{p4}).
Once $U=V(G)$, we have $s=|M|$, $d=0$, and $f=|A|-|M|$,
and (\ref{e1}) implies the stated lower bound on $|M|$.

We proceed to the description of the extension operations.
Therefore, suppose that $U$ is a proper subset of $V(G)$.
Since $G$ is connected, and some vertex in $B$ has degree less than $\Delta$,
some vertex $u$ in $B\setminus U$ has less than $\Delta$ neighbors in $A\setminus U$,
that is, if $d_{\bar{U}}(u)=|N_G(u)\setminus U|$, then $1\leq d_{\bar{U}}(u)\leq \Delta-1$, where the first inequality follows from property~(\ref{p3}).
We choose $u \in B \sm U$ such that $d_{\bar{U}}(u)$ is as small as possible.

%\medskip

\paragraph{ {\bf Case 1:}} $d_{\bar{U}}(u)=1$.\smallskip

%\medskip

\noindent Let $v$ be the unique neighbor of $u$ in $A\setminus U$.
Let $\{ u_1,\ldots,u_k\}$ be the set of all vertices $u$ in $B\setminus U$ with $N_G(u)\setminus U=\{ v\}$, and note that $1 \leq k \leq \Delta$.
Let $U'=U\cup \{ u_1,\ldots,u_k,v\}$.
For some integer $\ell \leq k$, we may assume that
$\{ u_1,\ldots,u_{\ell}\}$ is the set of those $u_i$ with $i\in [k]$
such that $u_i$ has a neighbor in $A\cap U$, and no neighbor of $u_i$ in $A\cap U$ is incident with $M$.
Note that every vertex $u_i$ with $i\in [k]\setminus [\ell]$
either has no neighbor in $A\cap U$
or has some neighbor in $A\cap U$ that is incident with $M$.

First, suppose that $\ell\geq 2$.
Let $M'$ arise from $M$ by adding, for every $i\in [\ell]$, an edge between $u_i$ and a neighbor $w_i$ of $u_i$ in $A\cap U$. Note that all these neighbors $w_i$ in $A \cap U$ are distinct. Indeed, if two vertices $u_i$ and $u_j$ have a common neighbor $w$ in $A \cap U$, then the set of vertices $\{v, u_i, u_j, w\}$ would induce a $C_4$ in $G$.
Note also that $M'$ is indeed a uniquely restricted matching, as if there exists an edge $u_iw_j$ with $i,j \in [\ell]$ and $i \neq j$ that could potentially create an $M'$-alternating cycle, then the set of vertices $\{v, u_i, u_j, w_j\}$ would again induce a $C_4$ in $G$. Clearly, replacing $(U,M)$ with $(U',M')$, we maintain properties (\ref{p0}) to (\ref{p3}), and $s'=s+\ell$.
Let $n_d$ be the number of vertices in $A\cap U$ that are not incident with an edge in $M'$,
have a neighbor in $B\setminus U$,
but do not have a neighbor in $B\setminus U'$; note that each such vertex has a neighbor in the set $\{u_1, \ldots, u_k\}$.
As every vertex in $\{u_1, \ldots, u_k\}$ is neighbor of $v$ and of a vertex incident with an edge in $M'$, it holds that $n_d\leq k(\Delta-2)\leq \Delta(\Delta-2)$.
%\jul{This $\Delta-2$ is not obvious, maybe we should add some explanation?}
If $v$ has a neighbor in $B\setminus U'$, then $d'=d-n_d+1$ and $f'=f+n_d$, and,
if $v$ has no neighbor in $B\setminus U'$, then $d'=d-n_d$ and $f'=f+n_d+1$.
In both cases $d'+f'=d+f+1$ and $f'\leq f+n_d+1$.
Since $\frac{(\Delta-1)^2}{\Delta-2}\Big((\Delta-2)\ell-1\Big)\geq \Delta(\Delta-2)+1\geq n_d+1$,
property (\ref{p4}) is maintained.

Next, suppose that $\ell\leq 1$.
Let $M'$ arise from $M$ by adding the edge $u_1v$.
Clearly, replacing $(U,M)$ with $(U',M')$, we maintain properties (\ref{p0}) to (\ref{p3}), and $s'=s+1$.
Defining $n_d$ exactly as above, we obtain
$n_d\leq k(\Delta-2)+\ell\leq \Delta(\Delta-2)+1$,
$d'=d-n_d$, and
$f'=f+n_d$.
Since $\frac{(\Delta-1)^2}{\Delta-2}(\Delta-2)\geq \Delta(\Delta-2)+1\geq n_d$,
property (\ref{p4}) is maintained.

\paragraph{ {\bf Case 2:}} $2\leq d_{\bar{U}}(u)\leq \Delta-1$.\smallskip

\noindent Let $\{ v_1,\ldots,v_k\}=N_G(u)\setminus U$ and let $U'=U\cup \{ u,v_1,\ldots,v_k\}$. Note that $2 \leq k \leq \Delta -1$.

First, suppose that $u$ has a neighbor $v$ in $A\cap U$, and that no neighbor of $u$ in $A\cap U$ is incident with $M$. Let $M'$ arise from $M$ by adding the edge $uv$.
Clearly, replacing $(U,M)$ with $(U',M')$, we maintain properties (\ref{p0}) to (\ref{p2}), and $s'=s+1$. Let us prove that property~(\ref{p3})~is also maintained. Since $G$ has no $C_4$ and $k\geq 2$, no vertex in $B \setminus U$ that is distinct from $u$ can have more than one neighbor among $v_1,\ldots,v_k$. Since we are in Case~2, every vertex in $B\setminus U$ has more than one neighbor in $A\setminus U$, hence property (d) remains true.
Similarly as above,
let $n_d$ be the number of vertices in $A\cap U$ that are not incident with an edge in $M'$,
have a neighbor in $B\setminus U$,
but do not have a neighbor in $B\setminus U'$.
Note that $n_d\leq \Delta-k-1$,
$d'=d+k-n_d-1$, and
$f'=f+n_d$.
Since $\frac{(\Delta-1)^2}{\Delta-2}\Big((\Delta-2)-(k-1)\Big)\geq \Delta-k-1\geq n_d$,
property (\ref{p4}) is maintained.

Next, suppose that $u$ either has no neighbor in $A\cap U$ or some neighbor of $u$ in $A\cap U$ is incident with $M$.
Let $M'$ arise from $M$ by adding the edge $uv_1$.
Clearly, replacing $(U,M)$ with $(U',M')$, we again maintain properties (\ref{p0}) to (\ref{p3}), and $s'=s+1$. Note that, in the case where $u$ has a neighbor in $A \cap U$, $v_1$ does not have neighbors in $M$ because of property~(\ref{p2}), which guarantees that $M'$ is indeed a uniquely restricted matching.
Defining $n_d$ exactly as above, we obtain
$n_d\leq \Delta-k-1$. Indeed, if $u$ has no neighbor in $A \cap U$, then $n_d = 0$. On the other hand, if $u$ has a neighbor in $A\cap U$ that is incident with $M$, then  $n_d\leq \Delta-k-1$. As $k \leq \Delta -1$, in both cases it holds that $n_d\leq \Delta-k-1$. Also, we get that $d'=d + k -n_d + 1$ and $f'=f+n_d$, and the same calculation as above implies that property (\ref{p4}) is maintained.

\medskip

\noindent Since the considered cases exhaust all possibilities,
and in each case we described an extension that maintains the relevant properties,
the proof is complete up to the running time of the algorithm, which we proceed to analyze. One can easily check that each extension operation takes time $O(\Delta n)$, where $n = n(G)$. As in each extension operation, the size of $U$ is incremented by at least one, it follows that the overall running time of the algorithm is $O(\Delta n^2)$.\end{proof}

With Lemma \ref{lemmanew1} at hand, we proceed to our first approximation algorithm.

\begin{theorem}\label{theoremnew1}
Let $\Delta \geq 3$ be an integer. For a given connected $C_4$-free bipartite graph $G$ of maximum degree at most $\Delta$,
one can find in polynomial time a uniquely restricted matching $M$ of $G$ of size at least
$\frac{(\Delta-1)^2+(\Delta-2)}{(\Delta-1)^3+(\Delta-2)}\nu_{ur}(G).$
\end{theorem}
\begin{proof}
Let $\alpha=\frac{(\Delta-1)^2+(\Delta-2)}{(\Delta-1)^3+(\Delta-2)}$ and let ${\cal G}$ be the set of all $C_4$-free bipartite graphs $G$ of maximum degree at most $\Delta$
such that every component of $G$ has a vertex of degree less than $\Delta$.
First, we prove that, for every given graph $G$ in ${\cal G}$,
one can find in polynomial time a uniquely restricted matching $M$ of size at least $\alpha\nu_{ur}(G)$.
Therefore, let $G$ be in ${\cal G}$.

If $G$ has a vertex $u$ of degree $1$, and $v$ is the unique neighbor of $u$, then let $G'=G-\{ u,v\}$.
Clearly, $\nu_{ur}(G')\geq \nu_{ur}(G)-1$,
and if $M'$ is a uniquely restricted matching of $G'$,
then $M'\cup \{ uv\}$ is a uniquely restricted matching of $G$.
Note that $G'$ belongs to ${\cal G}$.
Let $G''$ be the graph obtained from $G'$ by removing every isolated vertex.
%Moreover, if $G'$ contains a set $X$ of isolated vertices, then let $G'' = G' - X$.
Clearly, $\nu_{ur}(G'')\geq \nu_{ur}(G')$,
if $M''$ is a uniquely restricted matching of $G''$,
then $M''$ is a uniquely restricted matching of $G'$, and
$G''$ belongs to ${\cal G}$.

Iteratively repeating these reductions,
we efficiently obtain a set $M_1$ of edges of $G$ as well as a subgraph $G_2$ of $G$ such that
$G_2\in {\cal G}$,
$\nu_{ur}(G_2)\geq \nu_{ur}(G)-|M_1|$,
$M_1\cup M_2$ is a uniquely restricted matching of $G$ for every uniquely restricted matching $M_2$ of $G_2$, and
either $n(G_2)=0$ or $\delta(G_2)\geq 2$.
%\jul{The case where $G_2$ contains an isolated vertex is missing, isn't it?}
%I think we should explain what happens if $G_2$ contains an isolated vertex.}
Note that if $G$ has minimum degree at least $2$, then we may choose $M_1$ empty and $G_2$ equal to $G$.
Now, by suitably choosing the bipartition of each component $K$ of $G_2$,
and applying Lemma \ref{lemmanew1} to $K$,
one can determine in polynomial time a uniquely restricted matching $M_2$ of $G_2$
with $|M_2|\geq \alpha\nu_{ur}(G_2)$.
Since the set $M_1\cup M_2$ is a uniquely restricted matching of $G$ of size at least
$|M_1|+\alpha\nu_{ur}(G_2)\geq |M_1|+\alpha(\nu_{ur}(G)-|M_1|)\geq \alpha\nu_{ur}(G)$,
the proof of our claim about ${\cal G}$ is complete.

%Now, let $G$ be a given connected $C_4$-free bipartite graph $G$ of maximum degree at most $\Delta$.
%If $G\in{\cal G}$, the desired statement already follows.
%If $G\not\in{\cal G}$, that is, $G$ is $\Delta$-regular,
%then removing any vertex from $G$ results in a graph $G'$ in ${\cal G}$ with $\nu_ {ur}(G')\geq \nu_ {ur}(G)-1$.
%By the above,
%one can determine in polynomial time a uniquely restricted matching $M'$ of $G'$
%with $|M'|\geq \alpha\nu_{ur}(G')\geq \alpha(\nu_ {ur}(G)-1)$.
%Since $M'$ is a uniquely restricted matching of $G$,
%the desired statement follows.
%
%ALTERNATIVE STATEMENT:

Now, let $G$ be a given connected $C_4$-free bipartite graph $G$ of maximum degree at most $\Delta$. If $G$ is not $\Delta$-regular, then $G\in {\cal G}$, and the desired statement already follows. Hence, we may assume that $G$ is $\Delta$-regular, which implies that its two partite sets $A$ and $B$ are of the same order. By \cite{peraso}, we can efficiently decide whether $\nu_{ur}(G)=\nu(G)$. Furthermore, if $\nu_{ur}(G)=\nu(G)$, then, again by \cite{peraso}, we can efficiently determine a maximum matching that is uniquely restricted. Hence, we may assume that $\nu_{ur}(G)<\nu(G)$. This implies that $\nu_{ur}(G)<|A|$, and, hence, there is some vertex $u \in V(G)$  with $\nu_{ur}(G-u)=\nu_{ur}(G)$. Since $G-u\in {\cal G}$ for every vertex $u$ of $G$, considering the $n(G)$ induced subgraphs $G-u$ for $u\in V(G)$, one can determine in polynomial time a uniquely restricted matching $M$ of $G$ with $|M|\geq \max\{\alpha \nu_{ur}(G-u):u\in V(G)\} = \alpha\nu_{ur}(G)$. The desired statement follows.\end{proof}

\section{A $9/5$-approximation for subcubic bipartite graphs}
\label{ap:Delta3}

We show that -- at least for $\Delta=3$ --
$C_4$-freeness is not an essential assumption. Namely, this section is devoted to proving the following theorem.
%Because of space limitations, the proof is moved to Appendix~\ref{ap:Delta3}.

\begin{theorem}\label{th:approxDelta3}
For a given connected subcubic bipartite graph $G$,
one can find in polynomial time a uniquely restricted matching of $G$
of size at least $\frac{5}{9}\nu_{ur}(G)$.
\end{theorem}

We believe that Theorem \ref{th:approxDelta3} extends to larger maximum degrees,
that is, the conclusion of Theorem \ref{theoremnew1} should hold without the assumption of $C_4$-freeness.

%\jul{Change this intro.}
%In this section we provide a proof of Theorem~\ref{th:approxDelta3}.
%, which we reformulate for convenience as Theorem~\ref{th:approxDelta3bis}. In both Lemma~\ref{lemmanewDelta} and Theorem~\ref{th:approxDelta3bis}, we state the approximation ratios in terms of $\Delta$ even if they only apply to $\Delta =3$, in order to highlight the contribution of the degree that follows from the analysis, with the hope that it could be eventually generalized to higher degrees.
%Interestingly, note that the ratios
%$\frac{(\Delta-1)^2+(\Delta-2)}{(\Delta-1)^3+(\Delta-2)}$ and $\frac{(\Delta-1)^2+1}{2 (\Delta-1)^2+1}$ given by Theorem~\ref{theoremnew1} (which applies only to $C_4$-free bipartite graphs) and Theorem~\ref{th:approxDelta3bis}, respectively, coincide only for $\Delta = 3$.

\medskip

In order to ease the presentation,  in this section we will use figures to describe some of the ``patterns'' considered by the algorithms. More formally, given a graph $G$, a \emph{pattern} $P$ is a subgraph of $G$ in which the set of vertices that have neighbors in $V(G) \setminus V(P)$ is fixed.

%which we fix which vertices should be connected to the rest of the graph and each other vertex should not be connected to the rest of the graph.

In all these figures, the partition of the corresponding bipartite subcubic graph into two sets $A$ and $B$ is represented by using squares and circles, respectively. The half-edges specify which vertices in a pattern have neighbors outside of it.

%\jul{The half-edges show that the needed connections.}

%\jul{I am not satisfied with the following definition of pattern. Do you have some suggestion?}

%Given a graph $G$, a \emph{pattern} $H$ is a subgraph of $G$ on which we fix which vertices should be connected to the rest of the graph and each other vertex should not be connected to the rest of the graph.

The following lemma is crucial in order to prove Theorem~\ref{th:approxDelta3}; it plays a role similar to the one played by Lemma~\ref{lemmanew1} for proving Theorem~\ref{theoremnew1}.

\hspace{1cm}
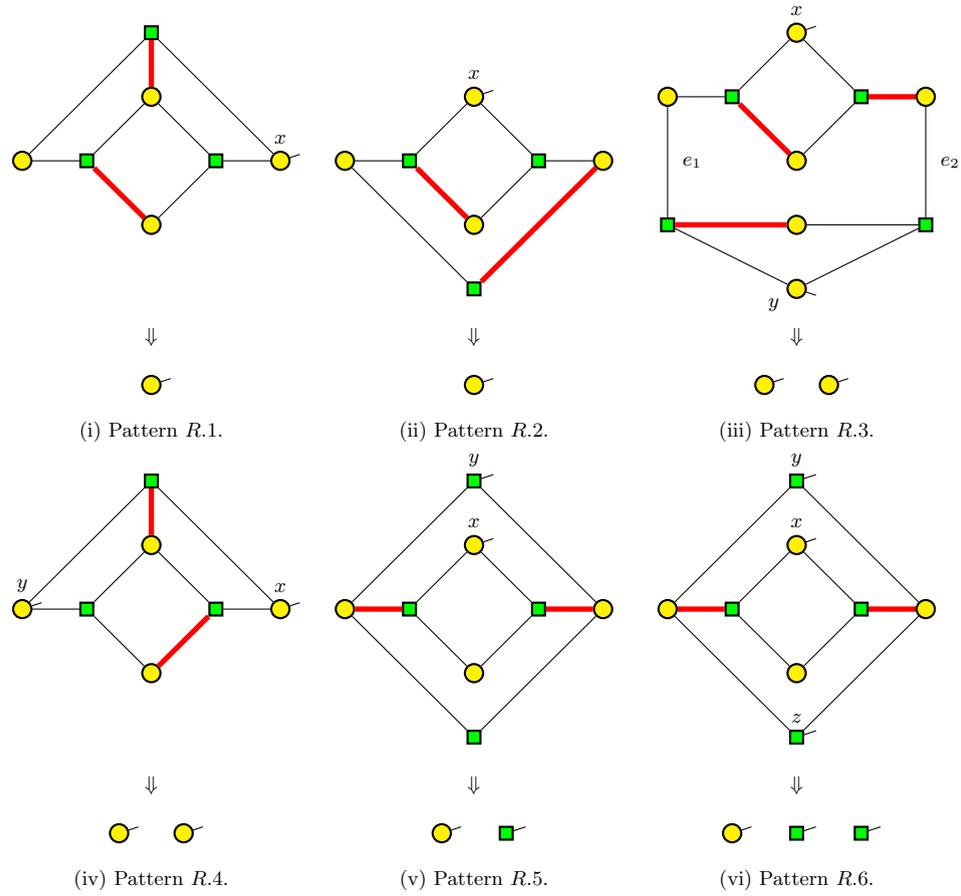
\begin{figure}[h]
\centering
  \begin{tikzpicture}[scale=0.85, every node/.style={transform shape}]
    \def\x{0}
    \def\y{0}
    \pc{0,0}{}

    \pa{1,1}{a1}
    \pa{1,-1}{a2}
    \pb{2,0}{b1}
    \paa{3,0}{a3}{label=$x$}
    \pa{-1,0}{a4}

    \pb{1,2}{b2}
    \draw (c) -- (a1);
    \draw (c) -- (a2) [select];
    \draw (c) -- (a4);
    \draw (b1) -- (a1);
    \draw (b1) -- (a2);
    \draw (b1) -- (a3);
    \draw (b2) -- (a1) [select];
    \draw (b2) -- (a4);
    \draw (b2) -- (a3);
    \outh{a3}
    \capt{0,-3}{$\Downarrow$};

    \pa{1,-3.5}{a5}
    \outh{a5}

    \capt{0+\x, -4.5+\y}{(i) Pattern $R.1$.};

    \def\x{5}
    \def\y{0}
    \pc{0,0}{}

    \paa{1,1}{a1}{label=$x$}
    \pa{1,-1}{a2}
    \pb{2,0}{b1}
    \pa{3,0}{a3}
    \pa{-1,0}{a4}

    \pb{1,-2}{b3}
    \draw (c) -- (a1);
    \draw (c) -- (a2) [select];
    \draw (c) -- (a4);
    \draw (b1) -- (a1);
    \draw (b1) -- (a2);
    \draw (b1) -- (a3);

    \draw (b3) -- (a4);
    \draw (b3) -- (a3) [select];

    \outh{a1}
    \capt{0+\x,-3+\y}{$\Downarrow$};

    \pa{1,-3.5}{a5}
    \outh{a5}

    \capt{0+\x, -4.5+\y}{(ii) Pattern $R.2$.};

    \def\x{10}
    \def\y{1 }
    \pc{0,0}{}

    \paa{1,1}{a1}{label=$x$}
    \pa{1,-1}{a2}
    \pb{2,0}{b1}
    \pa{3,0}{a3}
    \pa{-1,0}{a4}
    \pb{-1,-2}{b2}
    \pb{3,-2}{b3}
    \pa{1,-2}{a5}
    \paa{1,-3}{a6}{label=190:$y$}

    \draw (c) -- (a1);
    \draw (c) -- (a2) [select];
    \draw (c) -- (a4);
    \draw (b1) -- (a1);
    \draw (b1) -- (a2);
    \draw (b1) -- (a3) [select];
    \draw (a4) -- (b2) node [midway,label=0:$e_1$] {};
    \draw (a3) -- (b3) node [midway,label=0:$e_2$] {};
    \draw (b3) -- (a5);
    \draw (b3) -- (a6);
    \draw (b2) -- (a5) [select];
    \draw (b2) -- (a6);

    \outh{a1}
    \outb{a6}

    \capt{0+\x,-4+\y}{$\Downarrow$};

    \pa{0.5,-4.5}{a7}
    \pa{1.5,-4.5}{a8}
    \outh{a8}
    \outh{a7}
    \capt{0+\x, -5.5+\y}{(iii) Pattern $R.3$.};

    \def\x{0}
    \def\y{-7}
    \pc{0,0}{}

    \pa{1,1}{a1}
    \pa{1,-1}{a2}
    \pb{2,0}{b1}
    \paa{3,0}{a3}{label=$x$}
    \paa{-1,0}{a4}{label=$y$}

    \pb{1,2}{b2}
    \draw (c) -- (a1);
    \draw (c) -- (a2);
    \draw (c) -- (a4);
    \draw (b1) -- (a1);
    \draw (b1) -- (a2) [select];
    \draw (b1) -- (a3);
    \draw (b2) -- (a1) [select];
    \draw (b2) -- (a4);
    \draw (b2) -- (a3);
    \outh{a4}
    \outh{a3}

    \capt{0+\x,-3+\y}{$\Downarrow$};

    \pa{0.5,-3.5}{a7}
    \pa{1.5,-3.5}{a8}
    \outh{a8}
    \outh{a7}
    \capt{0+\x, -4.5+\y}{(iv) Pattern $R.4$.};

    \def\x{5}
    \def\y{-7}
    \pc{0,0}{}

    \paa{1,1}{a1}{label=$x$}
    \pa{1,-1}{a2}
    \pb{2,0}{b1}
    \pa{3,0}{a3}
    \pa{-1,0}{a4}

    \pbb{1,2}{b2}{label=$y$}
    \pb{1,-2}{b3}
    \draw (c) -- (a1);
    \draw (c) -- (a2);
    \draw (c) -- (a4) [select];
    \draw (b1) -- (a1);
    \draw (b1) -- (a2);
    \draw (b1) -- (a3) [select];
    \draw (b2) -- (a4);
    \draw (b2) -- (a3);

    \draw (b3) -- (a4);
    \draw (b3) -- (a3);

    \outh{a1}
    \outh{b2}

    \capt{0+\x,-3+\y}{$\Downarrow$};

    \pa{0.5,-3.5}{a7}
    \pb{1.5,-3.5}{a8}
    \outh{a8}
    \outh{a7}
    \capt{0+\x, -4.5+\y}{(v) Pattern $R.5$.};

    \def\x{10}
    \def\y{-7}
    \pc{0,0}{}

    \paa{1,1}{a1}{label=$x$}
    \pa{1,-1}{a2}
    \pb{2,0}{b1}
    \pa{3,0}{a3}
    \pa{-1,0}{a4}

    \pbb{1,2}{b2}{label=$y$}
    \pbb{1,-2}{b3}{label=$z$}
    \draw (c) -- (a1);
    \draw (c) -- (a2);
    \draw (c) -- (a4) [select];
    \draw (b1) -- (a1);
    \draw (b1) -- (a2);
    \draw (b1) -- (a3) [select];
    \draw (b2) -- (a4);
    \draw (b2) -- (a3);

    \draw (b3) -- (a4);
    \draw (b3) -- (a3);

    \outh{a1}
    \outh{b2}
    \outh{b3}

    \capt{0+\x,-3+\y}{$\Downarrow$};

    \pa{0,-3.5}{a7}
    \pb{1,-3.5}{a8}
    \pb{2,-3.5}{a9}
    \outh{a8}
    \outh{a7}
    \outh{a9}
    \capt{0+\x, -4.5+\y}{(vi) Pattern $R.6$.};

  \end{tikzpicture}

  \caption{The six forbidden patterns.
    % \ig{this caption should be updated, as in the current version there is no preprocessing anymore. Also, to be consistent with the other figures, please add a point at the end of each subcaption (in all figures)}
}
  \label{fig:redpat}
\end{figure}

\newpage
\begin{lemma}
  \label{lemmanewDelta}
  If $G$ is a connected subcubic bipartite graph with partite sets $A$ and $B$
  such that
  \begin{enumerate}[(1)]
  \item\label{c1} $G$ does not any of the patterns $R.1$, $R.2$, $R.3$, $R.4$, $R.5$, or $R.6$ depicted in Fig.~\ref{fig:redpat},
  \item\label{c2} $G$ does not contain two vertices with the same neighborhood,
  \item\label{c4}  each vertex of $G$ has degree at least $2$, and
  \item \label{c3} some vertex in $B$ has degree at most $2$,
  \end{enumerate}
  then a uniquely restricted matching $M$ of $G$
  of size at least
  $\frac{5}{9} \nu_{ur}(G)$ can be found in polynomial time.
\end{lemma}

\begin{proof}
  %We give an algorithmic proof of the lower bound such that the running time of the corresponding algorithm is polynomial in $n(G)$, which immediately implies the second part of the statement.
If $n(G) \leq 10$, then we solve the problem optimally by brute force (we will see that the largest pattern without neighbors outside of it considered in the proof has 10 vertices).
  Therefore, we assume henceforth that $G$ contains at least $11$ vertices.
  In the following, we look for a uniquely restricted matching $M$  of size at least $\frac{5}{9} |A|$.
  As $|A| \geq \nu_{ur}(G)$, this implies the desired result.
  We define two types of $C_4$, namely $C_4^1$ and $C_4^2$, as follows.
  A $C_4^1$ is a subgraph of $G$ isomorphic to a  $C_4$ such that if $V(C_4^1) \cap A = \{a_1,a_2\}$, then $d_G(a_1) = 3$ and $d_G(a_2) = 2$.
  A $C_4^2$ is a subgraph of $G$ isomorphic to a  $C_4$ such that if $V(C_4^2) \cap A = \{a_1,a_2\}$, then $d_G(a_1) = d_G(a_2) = 3$.
  Note that because of condition~(\ref{c2}), there is no subgraph $G'$ of $G$ isomorphic to $C_4$ such that if $V(G') \cap A = \{a_1,a_2\}$, then $d_G(a_1)  = d_G(a_2) = 2$, as it implies that $a_1$ and $a_2$ have the same neighborhood.

  Similarly to the proof of Lemma~\ref{lemmanew1}, throughout the execution of our algorithm we maintain a triple $(U,M, \gamma)$ that respects the following properties:

 % \ig{we should put labels to the properties to cite them appropriately, as in Lemma 1}
  %\jul{Done.}
  \begin{enumerate}[(a)]
  \item\label{qa} $U \subseteq V(G)$,
  \item\label{qb} $M$ is a uniquely restricted matching of $G$ such that $V(M) \subseteq U$,
  \item\label{qc} $\gamma: U \cap A \rightarrow \{\select, \drain, \forg\}$,
  \item\label{qd} for every $v \in A \cap U$, $\gamma(v) = \select$ if and only if $v \in V(M)$  and $\gamma(v) = \drain$ only if that $v$ has at least one neighbor in $B \sm U$,
  \item\label{qe} there is no edge between vertices in $U \cap B$ and vertices in $A \sm U$,
  \item\label{qf} every vertex in $B \sm U$ has at least one neighbor in $A \sm U$,
  \item\label{qg} if $s = |\{v \in A : \gamma(v) = \select\}|$, $d = |\{v \in A : \gamma(v) = \drain\}|$, and $f = |\{v \in A : \gamma(v) = \forg\}|$, then
    \begin{equation}
      \label{eq:delta}
      4\left(s-(d+f)\right) \geq f, \ \ \mathrm{ and}
    \end{equation}
  \item \label{qh} for each $C^1_4$ $G'$ in $G \sm U$, there is no vertex $v$ in $N_G(V(G')) \cap U \cap A$ such that $\gamma(v) = \drain$.
  \end{enumerate}

  \medskip

  We initialize the algorithm with $U = \emptyset$, $M = \emptyset$, and $\gamma = \varnothing$.
  Note that properties~(\ref{qa}) to (\ref{qh}) are satisfied.

  In the first part of the algorithm, we focus on removing the $C_4$'s from $G \sm U$.
  For this we first take care the $C_4^1$'s in $G \sm U$, and then we deal with the $C_4^2$'s.

  As long as  $G \sm U$ is not empty, we consider the first of the following cases such that the corresponding condition is fulfilled, where $d_{\bar{U}}(u)=|N_G(u)\setminus U|$:

  \begin{itemize}
  \item[$\bullet$]{Case 1: there exists $u \in B \sm U$ such that $d_{\bar{U}}(u) = 1$.}
  \item[$\bullet$] {Case 2: there exists $G'$, a $C_4^1$ in $G \sm U$.}
  \item[$\bullet$] {Case 3: there exists  $G'$, a $C_4^2$ in $G \sm U$.}
  \item[$\bullet$] {Case 4:  there exists $u \in B \sm U$ such that $d_{\bar{U}}(u) = 2$.}
  \end{itemize}

  Note that, by property~(\ref{qe}) and the connectivity of $G$, we know that, as long as $G \sm U$ is not the empty graph, at least one of these four cases should apply.
  We study each case and show that for each of them, we can find a new triple $(U',M', \gamma')$ starting from $(U,M,\gamma)$ such that
  % $U$ is a proper subset of $U'$ and $(U',M',\gamma')$ respects the properties (a) to (h).
  $U$ is a proper subset of $U'$,
  $M$ is a proper subset of $M'$,
  $\gamma$ is the restriction of $\gamma'$ to $U$, and
  properties (\ref{qa}) to (\ref{qh}) hold for $(U',M',\gamma')$,
  where $s'$, $d'$, and $f'$ denote the updated values considered in (\ref{qg}). Once $U=V(G)$, we have $s=|M|$, $d=0$, and $f=|A|-|M|$,
  and (\ref{eq:delta})
  implies that $|M| \geq     \frac{5}{9} |A|$.

  In order to prove that such a triple can indeed be found in polynomial time, we distinguish four cases.

   %\ig{we should properly define what \emph{resolving a pattern} means}
  In the following, by \emph{resolving} a pattern $P$ we mean that from a triple $(U,M,\gamma)$ that respects  properties~(\ref{qa}) to (\ref{qh}) such that $U \cap V(P) = \emptyset$, we exhibit a triple $(U',M',\gamma')$ that also respects properties~(\ref{qa}) to (\ref{qh}) and such that $U' = U \cup V(P)$.

%\ig{in Case 1 below, do we need to assume that there is no $C_4$, or not?}
 % \jul{No, I wrote the case again.}

  \paragraph{ {\bf Case 1:}} there exists $u \in B \sm U$ such that $d_{\bar{U}}(u) = 1$.\smallskip
  % We repeat the exact same operation than the Case 1 in Theorem~\ref{th:approxC4}.

  \noindent Assume that $d_{\bar{U}}(u) = 1$ and let $v$ be the only neighbor of $u$ in $A\sm U$.
  Let $\{ u_i : i \in [k]\}$ be the set of all vertices $u$ in $B\setminus U$ with $N_G(u)\setminus U=\{ v\}$.
  Note that $1 \leq k \leq 3$.
  Let $U'=U\cup \{v\} \cup \{u_i : i \in [k]\}$.
  By construction of $G$, we know that every vertex $u_i$, $i \in [k]$, is of degree at least $2$ in $G$, so it has at least one neighbor in $A \cap U$.
  For every $i \in [k]$, we define $W_i = N_G(u_i) \cap U$.
  Note that for $i,j \in [k]$ with $i \not = j$, $W_i$ and $W_j$ can intersect.
  Let $n_d$ be the number $|\{w \in \bigcup_{i \in [k]} W_i : \gamma(w) = \drain\}|$.
  Note that for each $u_i$, $i \in [k]$, $v \in N_G(u_i)$ and $v \not \in U$.
  This implies that for each $i \in [k]$, $|W_i| \leq 2$, and so, $n_d \leq 6$.

  First, assume that $n_d \leq 4$.
  Let $M'$ arise from $M$ by adding the edge $u_1v$.
  Let $\gamma'$ be obtained from $\gamma$ where $\gamma'(v) = \select$, and for each  $w \in \bigcup_{ i \in [k]} W_i$, such that $\gamma(w) = \drain$, then $\gamma'(w) = \forg$.
  Clearly, replacing $(U,M,\gamma)$ with $(U',M',\gamma')$, we maintain properties (\ref{qa}) to (\ref{qf}).
  By construction of $\gamma'$, we have that $s' = s+1$, $d' = d - n_d$, and $f' = f + n_d$.
  As $n_d \leq 4$, property (\ref{qg}) is maintained.
  Note that $\gamma'^{-1}(\drain) \subseteq\gamma^{-1}(\drain)$.
  This implies that property~(\ref{qh}) is maintained.

  Assume now that $n_d \geq 5$.
  This implies that $k = 3$, and two sets of $\{W_i: i \in [k]\}$, say $W_1$ and $W_2$, are such that
  $\{w \in W_1 \cup W_2 : \gamma(w)  \not = \drain\} = \emptyset$ and $|W_1| = |W_2| = 2$.
  Because of condition~(\ref{c2}), we know that we can find $w_1 \in W_1 \sm W_2$ and $w_2 \in W_2 \sm W_1$.
  Let $M'$ arise from $M$ by adding the edges $u_1w_1$ and $u_2w_2$.
  Let $\gamma'$ be obtained from $\gamma$ where $\gamma'(v) = \forg$, for each $w \in (\bigcup_{i \in [k]} W_i) \sm \{w_1,w_2\}$, such that $\gamma(w) = \drain$, then $\gamma'(w) = \forg$ and $\gamma'(w_1) = \gamma'(w_2) = \select$.
  Note that $M'$ is a uniquely restricted matching.
  Clearly, replacing $(U,M,\gamma)$ with $(U',M',\gamma')$, we maintain properties (\ref{qa}) to (\ref{qf}).
  % Let $n_d$ be the number $|\{w \in \bigcup_{i \in [k]} W_i : \gamma(w) = \drain\}|$.
  Then by construction of $\gamma'$ we obtain that $s' = s + 2$, $d' = d - n_d$, and $f' = f + n_d -2 +1 $.
  We obtain that $s'-(d'+f') = s  - (d+f) +3$.
  As, $n_d \leq 6$,
  property~(\ref{qg}) is maintained.
  Note that $\gamma'^{-1}(\drain) \subseteq\gamma^{-1}(\drain)$.
  This implies that property~(\ref{qh}) is maintained.

  % \ig{comment: it is not ``the property (\ref{qg})'', but ``property (\ref{qg})''. I tried to correct this everywhere, I hope I did not forget anyone}

%\newpage

  \paragraph{ {\bf Case 2:}} there exists a $C_4^1$ $G'$ in $G \sm U$.\smallskip

  \noindent We assume in this case that for every $u \in B \sm U$, $d_{\bar{U}}(u) \geq 2$.
  Let $V(G') \cap A = \{a_1,a_2\}$ such that $d_G(a_1) = 3$ and $V(G') \cap B = \{b_1,b_2\}$.
  For this case, when we say that  we define $M'$ and $\gamma'$ by updating the values of $M$ and $\gamma$ according to a figure, it means that
  we add to $M$ every red edge of the figure and for every $v$ of $A$ that is depicted in the figure, then $\gamma'(v) = \select$ if $v$ is the endpoint of a red edge, and $\gamma'(v) = \forg$ otherwise.

  First, assume that $N_G(b_1)\sm U = N_G(b_2)\sm U = \{a_1,a_2\}$.
  If the only vertex of $N_G(a_1) \sm V(G')$ is inside another $C_4$, then we are in the situation depicted in Fig.~\ref{fig:case2_1}(ii) and we define $U'$ to be the union of $U$ and the vertices of the two $C_4$'s, and we define $M'$ and $\gamma'$ by updating the values of $M$ and $\gamma$ according to Fig.~\ref{fig:case2_1}(ii).
  Otherwise, we are in the situation depicted in Fig.~\ref{fig:case2_1}(i), we define $U'$ to be $U \cup V(G')$, and we define $M'$ and $\gamma'$ by updating the values of $M$ and $\gamma$ according to Fig.~\ref{fig:case2_1}(i).
  In both cases, one can see that properties (\ref{qa}) to (\ref{qh}) are maintained.
  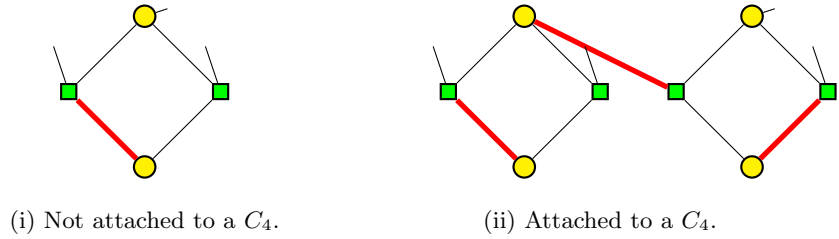
\begin{figure}[h]
   \centering
    \begin{tikzpicture}[scale=1, every node/.style={transform shape}]
      \def\x{0}
      \def\y{2}
      \pc{0,0}{}

      \pa{1,1}{a1}
      \pa{1,-1}{a2}
      \pb{2,0}{b1}
      \draw (c) -- (a1);
      \draw (c) -- (a2) [select];
      \draw (b1) -- (a1);
      \draw (b1) -- (a2);

      \outh{a1}
      \outtt{c}
      \outtt{b1}

      \capt{\x,-2+\y}{(i) Not attached to a $C_4$.};

      \def\x{5}
      \def\y{2}
      \pc{0,0}{}

      \pa{1,1}{a1}
      \pa{1,-1}{a2}
      \pb{2,0}{b1}
      \draw (c) -- (a1);
      \draw (c) -- (a2) [select];
      \draw (b1) -- (a1);
      \draw (b1) -- (a2);

      \def\x{8}
      \def\y{2}

      \pb{0,0}{ac}
      \pa{1,1}{aa1}
      \pa{1,-1}{aa2}
      \pb{2,0}{ab1}
      \draw (ac) -- (aa1);
      \draw (ac) -- (aa2);
      \draw (ab1) -- (aa1);
      \draw (ab1) -- (aa2) [select];

      \draw (a1) -- (ac) [select];

      \outh{aa1}
      \outtt{c}
      \outtt{b1}
      \outtt{ab1}
      \capt{-2+\x,-2+\y}{(ii) Attached to a $C_4$.};

    \end{tikzpicture}

    \caption{First case of $C_4^1$.}
    \label{fig:case2_1}
  \end{figure}

  Secondly, assume that there exists $a_3 \in A  \sm U$ such that $a_3 \in N_G(b_2)$ and $a_3 \not \in N_G(b_1)$.
  If there exists $b_3$ with neighbors only in $U \cup V(G) \cup \{a_3\}$, then we define $U'$ to be the union of $U$ and $\{a_1,a_2,a_3,b_1,b_2,b_3\}$, and we define $M'$ and $\gamma'$ by updating the values of $M$ and $\gamma$ according to Fig.~\ref{fig:case2_2}(ii).
  Otherwise, we are in the situation depicted in Fig.~\ref{fig:case2_2}(i), we define $U'$ to be $U \cup V(G') \cup \{a_3\}$, and we define $M'$ and $\gamma'$ by updating the values of $M$ and $\gamma$ according to Fig.~\ref{fig:case2_2}(i).
  In both cases we can see that properties (\ref{qa}) to (\ref{qh}) are maintained.
  \begin{figure}[h]
    \centering
    \begin{tikzpicture}[scale=1, every node/.style={transform shape}]
      \def\x{0}
      \def\y{2}
      \pc{0,0}{}

      \pa{1,1}{a1}
      \pa{1,-1}{a2}
      \pb{2,0}{b1}
      \pa{3,0}{a3}
      \draw (c) -- (a1);
      \draw (c) -- (a2) [select];
      \draw (b1) -- (a1);
      \draw (b1) -- (a2);
      \draw (b1) -- (a3) [select];

      \outh{a1}
      \outh{a3}
      \outtt{c}

      \capt{0.5+\x,-2+\y}{(i) If $b_3$ does not exist.};

      \def\x{6}
      \def\y{2}
      \pc{0,0}{}

      \pa{1,1}{a1}
      \pa{1,-1}{a2}
      \pb{2,0}{b1}
      \pa{3,0}{a3}
      \pb{4,1}{b2}

      \draw (c) -- (a1);
      \draw (c) -- (a2) [select];
      \draw (b1) -- (a1);
      \draw (b1) -- (a2);
      \draw (b1) -- (a3) [select];

      \draw (b2) -- (a1);
      \draw (b2) -- (a3);

      \outtt{c}

      \capt{0.5+\x,-2+\y}{(ii) If $b_3$ exists.};

    \end{tikzpicture}

    \caption{Second case of $C_4^1$.}
    \label{fig:case2_2}
  \end{figure}
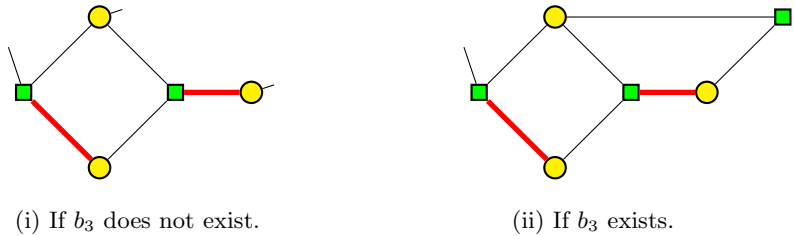

  \begin{figure}[H]
    \centering
    \begin{tikzpicture}[scale=0.9, every node/.style={transform shape}]
      \def\x{0}
      \def\y{2}
      \pc{0,0}{}

      \pa{1,1}{a1}
      \pa{1,-1}{a2}
      \pb{2,0}{b1}
      \pa{3,0}{a3}
      \pa{-1,0}{a4}
      \draw (c) -- (a1);
      \draw (c) -- (a2) [select];
      \draw (c) -- (a4);
      \draw (b1) -- (a1);
      \draw (b1) -- (a2);
      \draw (b1) -- (a3) [select];
      \outh{a1}
      \outh{a4}
      \outh{a3}

      \capt{0+\x, -2+\y}{Pattern (i).};

      \def\x{5}
      \def\y{-3}
      \pc{0,0}{}

      \pa{1,1}{a1}
      \pa{1,-1}{a2}
      \pb{2,0}{b1}
      \pa{3,0}{a3}
      \pa{-1,0}{a4}

      \pb{1,2}{b2}
      \draw (c) -- (a1);
      \draw (c) -- (a2);
      \draw (c) -- (a4);
      \draw (b1) -- (a1);
      \draw (b1) -- (a2);
      \draw (b1) -- (a3);
      \draw (b2) -- (a1);
      \draw (b2) -- (a4);
      \draw (b2) -- (a3);

      \outh{a3}

      \capt{0+\x, -1.7+\y}{Pattern (iv).};

      \def\x{10}
      \def\y{-3}
      \pc{0,0}{}

      \pa{1,1}{a1}
      \pa{1,-1}{a2}
      \pb{2,0}{b1}
      \pa{3,0}{a3}
      \pa{-1,0}{a4}

      \pb{1,2}{b2}
      \draw (c) -- (a1);
      \draw (c) -- (a2);
      \draw (c) -- (a4);
      \draw (b1) -- (a1);
      \draw (b1) -- (a2);
      \draw (b1) -- (a3);
      \draw (b2) -- (a1);
      \draw (b2) -- (a4);
      \draw (b2) -- (a3);
      \outh{a4}
      \outh{a3}
      \capt{0+\x, -1.7+\y}{Pattern (v).};

      \def\x{6}
      \def\y{2}
      \pc{0,0}{}

      \pa{1,1}{a1}
      \pa{1,-1}{a2}
      \pb{2,0}{b1}
      \pa{3,0}{a3}
      \pa{-1,0}{a4}

      \pb{1,2}{b2}
      \pb{1,-2}{b3}
      \draw (c) -- (a1);
      \draw (c) -- (a2) [select];
      \draw (c) -- (a4);
      \draw (b1) -- (a1);
      \draw (b1) -- (a2);
      \draw (b1) -- (a3);
      \draw (b2) -- (a1) [select];
      \draw (b2) -- (a4);
      \draw (b2) -- (a3);

      \draw (b3) -- (a4);
      \draw (b3) -- (a3) [select];

      \capt{0+\x, -2.7+\y}{Pattern (ii).};

      \def\x{0}
      \def\y{-3}
      \pc{0,0}{}

      \pa{1,1}{a1}
      \pa{1,-1}{a2}
      \pb{2,0}{b1}
      \pa{3,0}{a3}
      \pa{-1,0}{a4}

      \pb{1,2}{b2}
      \draw (c) -- (a1) [select];
      \draw (c) -- (a2);
      \draw (c) -- (a4);
      \draw (b1) -- (a1);
      \draw (b1) -- (a2);
      \draw (b1) -- (a3);
      \draw (b2) -- (a1);
      \draw (b2) -- (a4);
      \draw (b2) -- (a3) [select];

      \capt{0+\x, -1.7+\y}{Pattern (iii).};

      \def\x{0}
      \def\y{-6}
      \pc{0,0}{}

      \pa{1,1}{a1}
      \pa{1,-1}{a2}
      \pb{2,0}{b1}
      \pa{3,0}{a3}
      \pa{-1,0}{a4}

      \pb{1,-2}{b3}
      \draw (c) -- (a1);
      \draw (c) -- (a2);
      \draw (c) -- (a4);
      \draw (b1) -- (a1);
      \draw (b1) -- (a2);
      \draw (b1) -- (a3);

      \draw (b3) -- (a4);
      \draw (b3) -- (a3);

      \outh{a1}
      \capt{0+\x, -2.7+\y}{Pattern (vi).};

      \def\x{5}
      \def\y{-6}
      \pc{0,0}{}

      \pa{1,1}{a1}
      \pa{1,-1}{a2}
      \pb{2,0}{b1}
      \pa{3,0}{a3}
      \pa{-1,0}{a4}

      \pb{1,-2}{b3}
      \draw (c) -- (a1);
      \draw (c) -- (a2) [select];
      \draw (c) -- (a4);
      \draw (b1) -- (a1);
      \draw (b1) -- (a2);
      \draw (b1) -- (a3);

      \draw (b3) -- (a4) [select];
      \draw (b3) -- (a3);

      \outh{a1}
      \outh{a3}

      \capt{0+\x, -2.7+\y}{Pattern (vii).};

      \def\x{10}
      \def\y{-6}
      \pc{0,0}{}

      \pa{1,1}{a1}
      \pa{1,-1}{a2}
      \pb{2,0}{b1}
      \pa{3,0}{a3}
      \pa{-1,0}{a4}

      \pb{1,-2}{b3}
      \draw (c) -- (a1);
      \draw (c) -- (a2) [select];
      \draw (c) -- (a4);
      \draw (b1) -- (a1);
      \draw (b1) -- (a2);
      \draw (b1) -- (a3);

      \draw (b3) -- (a4) [select];
      \draw (b3) -- (a3);

      \outh{a1}
      \outh{a4}
      \outh{a3}
      \capt{0+\x, -2.7+\y}{Pattern (viii).};

      \def\x{0}
      \def\y{-11}
      \pc{0,0}{}

      \pa{1,1}{a1}
      \pa{1,-1}{a2}
      \pb{2,0}{b1}
      \pa{3,0}{a3}
      \pa{-1,0}{a4}

      \pb{1,2}{b2}
      \pb{1,-2}{b3}
      \draw (c) -- (a1);
      \draw (c) -- (a2);
      \draw (c) -- (a4);
      \draw (b1) -- (a1);
      \draw (b1) -- (a2);
      \draw (b1) -- (a3);
      \draw (b2) -- (a4);
      \draw (b2) -- (a3);

      \draw (b3) -- (a4);
      \draw (b3) -- (a3);

      \outh{a1}
      \outh{b2}

      \capt{0+\x, -2.7+\y}{Pattern (ix).};

      \def\x{5}
      \def\y{-11}
      \pc{0,0}{}

      \pa{1,1}{a1}
      \pa{1,-1}{a2}
      \pb{2,0}{b1}
      \pa{3,0}{a3}
      \pa{-1,0}{a4}

      \pb{1,2}{b2}
      \pb{1,-2}{b3}
      \draw (c) -- (a1);
      \draw (c) -- (a2);
      \draw (c) -- (a4);
      \draw (b1) -- (a1);
      \draw (b1) -- (a2) [select];
      \draw (b1) -- (a3);
      \draw (b2) -- (a1);
      \draw (b2) -- (a3) [select];

      \draw (b3) -- (a4) [select];
      \draw (b3) -- (a3);

      \capt{0+\x, -2.7+\y}{Pattern (x).};

      \def\x{10}
      \def\y{-11}
      \pc{0,0}{}

      \pa{1,1}{a1}
      \pa{1,-1}{a2}
      \pb{2,0}{b1}
      \pa{3,0}{a3}
      \pa{-1,0}{a4}

      \pb{1,2}{b2}
      \pb{1,-2}{b3}
      \draw (c) -- (a1);
      \draw (c) -- (a2);
      \draw (c) -- (a4);
      \draw (b1) -- (a1);
      \draw (b1) -- (a2) [select];
      \draw (b1) -- (a3);
      \draw (b2) -- (a1);
      \draw (b2) -- (a3) [select];

      \draw (b3) -- (a4) [select];
      \draw (b3) -- (a3);

      \outh{a4}

      \capt{0+\x, -2.7+\y}{Pattern (xi).};

      \def\x{0}
      \def\y{-16}
      \pc{0,0}{}

      \pa{1,1}{a1}
      \pa{1,-1}{a2}
      \pb{2,0}{b1}
      \pa{3,0}{a3}
      \pa{-1,0}{a4}

      \pb{1,2}{b2}
      \draw (c) -- (a1);
      \draw (c) -- (a2) [select];
      \draw (c) -- (a4);
      \draw (b1) -- (a1);
      \draw (b1) -- (a2);
      \draw (b1) -- (a3) [select];
      \draw (b2) -- (a1);
      \draw (b2) -- (a4) [select];
      \outh{a3}

      \capt{0+\x, -1.7+\y}{Pattern (xii).};

      \def\x{5}
      \def\y{-16}
      \pc{0,0}{}

      \pa{1,1}{a1}
      \pa{1,-1}{a2}
      \pb{2,0}{b1}
      \pa{3,0}{a3}
      \pa{-1,0}{a4}

      \pb{1,2}{b2}

      \draw (c) -- (a1);
      \draw (c) -- (a2) [select];
      \draw (c) -- (a4);
      \draw (b1) -- (a1);
      \draw (b1) -- (a2);
      \draw (b1) -- (a3) [select];
      \draw (b2) -- (a1);
      \draw (b2) -- (a4) [select];
      \outh{a4}
      \outh{a3}
      \capt{0+\x, -1.7+\y}{Pattern (xiii).};

    \end{tikzpicture}

    \caption{Third case of $C_4^1$.}
    \label{fig:case2_3}
  \end{figure}
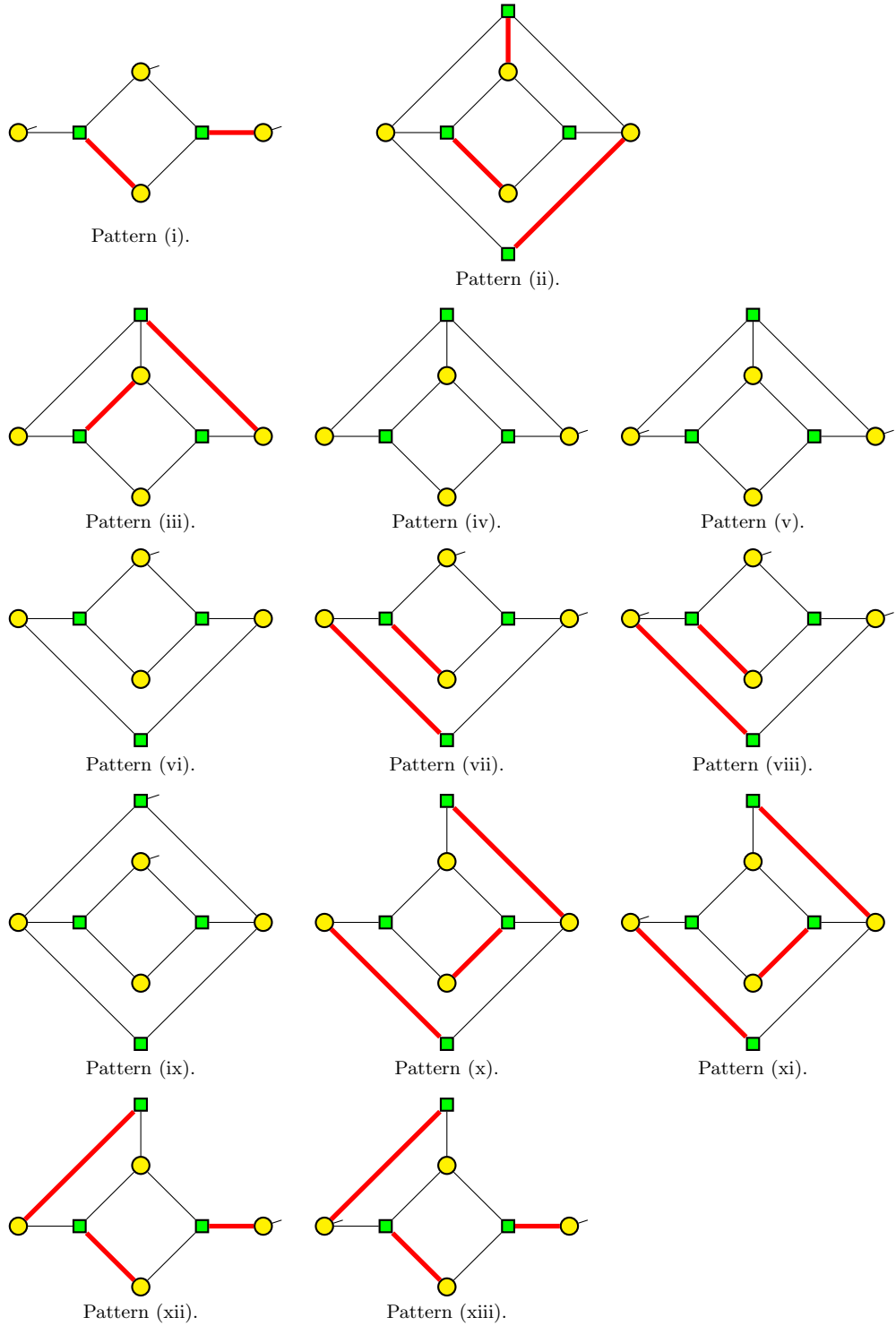

\newpage

  Third, assume that there exists $a_3$ and $a_4$ in $A  \sm U$ such that $a_3 \in N_G(b_2)$, $a_3 \not \in N_G(b_1)$, $a_4 \in N_G(b_1)$, and $a_4 \not \in N_G(b_2)$.
  This case is the most involved one, as many subcases have to be considered.
  Namely, we need to take care that there is no vertex in $B \sm U'$ of degree $0$ in $G \gm U'$ and to make sure that property (\ref{qh}) is maintained.
  In order to reduce the number of subcases, we sometimes do not take into consideration some vertex $u$ of $B$ that became of degree $0$ in $G \gm U$ after the update of $U'$ in the cases where property (\ref{qg}) is still maintained with the value of $s''$, $f''$, and $d''$ such that $d'' = d' -2$ and $f'' = f'+2$.
  This excludes the case where $u$ is connected to $\{v_1,v_2,v_3\}$ such that $\gamma'(v_i) = \drain$ for every $i \in [3]$ and then we can safely define
  $U'' = U' \cup \{u\}$,  $M'' = M' \cup \{\{u,v_1\}\}$, and $\gamma''(v) = \gamma'(v)$ for every $v \in (A \cap U') \sm \{v_1,v_2,v_3\}$, $\gamma''(v_1) = \select$, and $\gamma''(v_2) = \gamma''(v_3) = \forg$.
  The same applies to the case where $u$ is connected to $\{v_1,v_2,v_3\}$ such that there exists $i \in [3]$ such that $\gamma'(v_i) \not = \drain$.
  In this case, we only add $u$ to $U'$ without adding edges to $M'$, but for every $i \in [3]$ such that $\gamma'(v_i) = \drain$, then $\gamma'(v_i) = \forg$.
  This last condition is necessary in order to make sure that property (\ref{qd}) is maintained.
  In both cases, we will ignore these vertices of $B$ in the analysis, but we need to keep in mind that we need to add them each time one of these cases appears.
  We also sometime forget the third neighbor of a vertex of $A$ whenever its existence does not change how to resolve the pattern.

  Let $K = \{a_1,a_2,a_3,a_4,b_1,b_2\}$.
  As there are at most five edges between $K$ and $V(G) \sm (U \cup K)$, at most two vertices of $B$ can become of degree $0$.
  We depict in Fig.~\ref{fig:case2_3} every possible way these vertices can be connected to $K$ together with the possible extra edge from $K$ to $V(G) \sm (U \cup K)$.
  First note that Patterns (ii), (iii), and (x) have no neighbor outside of the pattern and contain less than $10$ vertices, so we have already resolved these patterns.
  Note also that Patterns (iv), (v), (vi), and (ix) correspond to patterns of condition~(\ref{c1}) and so are not in $G$.
  Moreover, Patterns (xi), (xii), and (xiii) do not need to have vertices labeled $\drain$ in order to define the triple $(U', M', \gamma')$ that incorporates the corresponding pattern to the part already treated and that respects properties (\ref{qa}) to (\ref{qh}).
  Therefore, there are only three remaining patterns to resolve, namely (i), (vii), and (viii).
  Note that in these three cases, there is at most one extra vertex of $B$.

  In the following, we focus our attention on Pattern (i) but the same arguments apply to Patterns  (vii) and (viii).
  As discussed above, we need to take care of property (\ref{qh}).
  If the pattern has no neighbor inside a $C_4$, then we can extend the triple $(U,M,\gamma)$ where the vertices of $A$ of this pattern that will not be labeled $\select$ are labeled $\drain$.
  Note that in this case, properties (\ref{qa}) to (\ref{qh}) are maintained.
  Assume that there is a $C_4$ such that exactly one vertex of this $C_4$ is a neighbor of a vertex of $K$.
  Then we are in on of the cases depicted in Fig.~\ref{fig:case2_4}, and we can extend the triple $(U,M,\gamma)$.
  Assume now that it is not the case and there is a $C_4$ with two vertices of this $C_4$ that have a neighbor in $K$.
  Then we are in one of the cases depicted in Fig.~\ref{fig:case2_5}.
  For  Pattern (xviii), it cannot exist because of condition~(\ref{c1}).
  For Pattern (xvi), we can extend the triple $(U,M,\gamma)$ according to the figure.
  For  Pattern (xvii), we assume
  that
  for both $C_4$'s of  this pattern,
  we cannot have either Pattern (xvi) or any of the patterns of Fig~\ref{fig:case2_4}.
%  \ig{why can we assume this?}
  {Otherwise, we start by solving one of these patterns.}
    %\jul{``For  Pattern (xvii), we assume that for both $C_4$ of this pattern, there is no $C_4$ with exactly one vertex of this $C_4$ is a neighbor of a vertex of $K$ and no Pattern (xvi) either'' $\rightarrow$
%    ``For  Pattern (xvii), we assume that
%    for both $C_4$'s of  Pattern (xvii),
%    we cannot have find the Patterns of Fig~\ref{fig:case2_4} or  Pattern (xvi).''
%  }
%  \ig{this previous sentence is not clear, please reformulate}.
%  \jul{just remove ``Then the vertices of $A$ in this pattern do not have a neighbor outside the pattern that are inside a $C_4$ ``}
%  \ig{which vertices are inside a $C_4$? Those in $A$, or the vertices outside the pattern?}.
  This implies that we can extend the triple $(U,M,\gamma)$ where the vertices of $A$ of the pattern that will not be labeled $\select$ are labeled $\drain$, and still respect property~(\ref{qh}).

  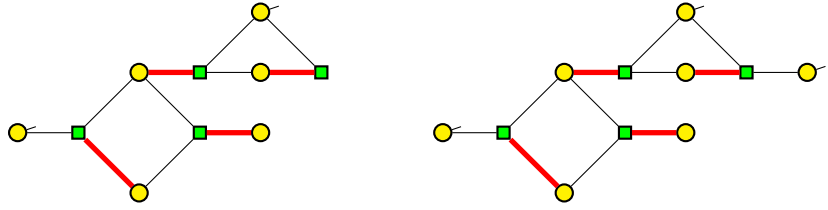
\begin{figure}[H]
    \centering
    \begin{tikzpicture}[scale=0.8, every node/.style={transform shape}]
      \def\x{0}
      \def\y{-14}
      \pc{0,0}{}

      \pa{1,1}{a1}
      \pa{1,-1}{a2}
      \pb{2,0}{b1}
      \pa{3,0}{a3}
      \pa{-1,0}{a4}

      \pb{2,1}{b5}
      \pa {3,2}{a7}
      \pa {3,1}{a8}
      \pb{4,1}{b6}
      % \pa {5,1}{a9}

      \draw (a1) -- (b5) [select];
      \draw (a7) -- (b5);
      \draw (a8) -- (b5);
      \draw (a7) -- (b6);
      \draw (a8) -- (b6) [select];

      \draw (c) -- (a1);
      \draw (c) -- (a2) [select];
      \draw (c) -- (a4);
      \draw (b1) -- (a1);
      \draw (b1) -- (a2);
      \draw (b1) -- (a3) [select];

      \outh{a7}
      % \outh{a9}
      \outh{a4}
      % \outb{a6}

      \def\x{7}
      \def\y{-14}
      \pc{0,0}{}

      \pa{1,1}{a1}
      \pa{1,-1}{a2}
      \pb{2,0}{b1}
      \pa{3,0}{a3}
      \pa{-1,0}{a4}

      \pb{2,1}{b5}
      \pa {3,2}{a7}
      \pa {3,1}{a8}
      \pb{4,1}{b6}
      \pa {5,1}{a9}

      \draw (a1) -- (b5) [select];
      \draw (a7) -- (b5);
      \draw (a8) -- (b5);
      \draw (a7) -- (b6);
      \draw (a8) -- (b6) [select];

      \draw (c) -- (a1);
      \draw (c) -- (a2) [select];
      \draw (c) -- (a4);
      \draw (b1) -- (a1);
      \draw (b1) -- (a2);
      \draw (b1) -- (a3) [select];
      \draw (b6) -- (a9);

      \outh{a7}
      \outh{a9}
      \outh{a4}
      % \outb{a6}

    \end{tikzpicture}

    \caption{Fourth case of $C_4^1$.}
    \label{fig:case2_4}
  \end{figure}

  \begin{figure}[H]
  \vspace{-.5cm}
    \centering
    \begin{tikzpicture}[scale=0.9, every node/.style={transform shape}]
      \def\x{2.5}
      \def\y{2}
      \pc{0,0}{}

      \pa{1,1}{a1}
      \pa{1,-1}{a2}
      \pb{2,0}{b1}
      \pa{3,0}{a3}
      \pa{-1,0}{a4}

      \pb{4,1}{b2}
      \pb{4,-1}{b3}
      \pa{4,0}{a5}
      \pa{5,0}{a6}

      \draw (c) -- (a1);
      \draw (c) -- (a2);
      \draw (c) -- (a4) [select];
      \draw (b1) -- (a1);
      \draw (b1) -- (a2) [select];
      \draw (b1) -- (a3);
      \draw (a1) -- (b2) [select];
      \draw (a3) -- (b3) [select];
      \draw (b3) -- (a5);
      \draw (b3) -- (a6);
      \draw (b2) -- (a5);
      \draw (b2) -- (a6);

      \outh{a4}
      \outh{a6}

      \capt{1+\x, -2+\y}{Pattern (xvi).};

      \def\x{0}
      \def\y{-2}
      \pc{0,0}{}

      \pa{1,1}{a1}
      \pa{1,-1}{a2}
      \pb{2,0}{b1}
      \pa{3,0}{a3}
      \pa{-1,0}{a4}
      \pb{-1,-2}{b2}
      \pb{3,-2}{b3}
      \pa{1,-2}{a5}
      \pa{1,-3}{a6}

      \draw (c) -- (a1);
      \draw (c) -- (a2) [select];
      \draw (c) -- (a4);
      \draw (b1) -- (a1);
      \draw (b1) -- (a2);
      \draw (b1) -- (a3) [select];
      \draw (a4) -- (b2);
      \draw (a3) -- (b3);
      \draw (b3) -- (a5) [select];
      \draw (b3) -- (a6);
      \draw (b2) -- (a5);
      \draw (b2) -- (a6);

      \outh{a1}
      \outh{a4}

      \outb{a6}
      \capt{0+\x, -4+\y}{Pattern (xvii).};

      \def\x{7}
      \def\y{-2}
      \pc{0,0}{}

      \pa{1,1}{a1}
      \pa{1,-1}{a2}
      \pb{2,0}{b1}
      \pa{3,0}{a3}
      \pa{-1,0}{a4}
      \pb{-1,-2}{b2}
      \pb{3,-2}{b3}
      \pa{1,-2}{a5}
      \pa{1,-3}{a6}

      \draw (c) -- (a1);
      \draw (c) -- (a2);
      \draw (c) -- (a4);
      \draw (b1) -- (a1);
      \draw (b1) -- (a2);
      \draw (b1) -- (a3);
      \draw (a4) -- (b2);
      \draw (a3) -- (b3);
      \draw (b3) -- (a5);
      \draw (b3) -- (a6);
      \draw (b2) -- (a5);
      \draw (b2) -- (a6);

      \outh{a1}
      \outb{a6}

      \capt{0+\x, -4+\y}{Pattern (xviii).};

    \end{tikzpicture}
    \caption{Fifth case of $C_4^1$.
      % \ig{please center the first subfigure}
    }
    \label{fig:case2_5}
  \end{figure}

  \paragraph{{\bf Case 3:}} there exists a $C_4^2$ $G'$ in $G \sm U$.\smallskip

 \noindent We assume in this case that for every $u \in B \sm U$, $d_{\bar{U}}(u) \geq 1$ and there is no $C_4^1$ in $G \gm U$.
  In particular, this implies that Case~2 will not occur anymore, and therefore, in the following property (\ref{qh}) will always be maintained.
  Thus, now we only focus on properties (\ref{qa}) to (\ref{qg}).
  Let $V(G') \cap A = \{a_1,a_2\}$ and $V(G') \cap B = \{b_1,b_2\}$.
  For this case, when we say that  we define $M'$ and $\gamma'$ by updating the values of $M$ and $\gamma$ according to a figure, it means that
  we add to $M$ every red edge of the figure and for every $v$ of $A$ that is depicted in the figure, then $\gamma'(v) = \select$ if $v$ is either the endpoint of a red edge or $v \in U$ and $\gamma(v) = \select$,
  otherwise either $v$ has a neighbor outside of the pattern depicted in the figure, and thus $\gamma'(v) = \drain$, or it does not and thus $\gamma'(v) = \forg$.

  First, assume that $N_G(b_1) = N_G(b_2) = \{a_1,a_2\}$.
  In this case, we need to distinguish between the labels of the vertices of  $(N_G(b_1) \cup N_G(b_2)) \cap U$.
  There are four possibilities depicted in Fig.~\ref{fig:case3_1}.
  We define $U'$ to be the union of $U$ and the vertices of $G'$, and we define $M'$ and $\gamma'$ by updating the values of $M$ and $\gamma$ according to Fig.~\ref{fig:case3_1}.
  One can check that, in each case, properties (\ref{qa}) to (\ref{qg}) are maintained.

  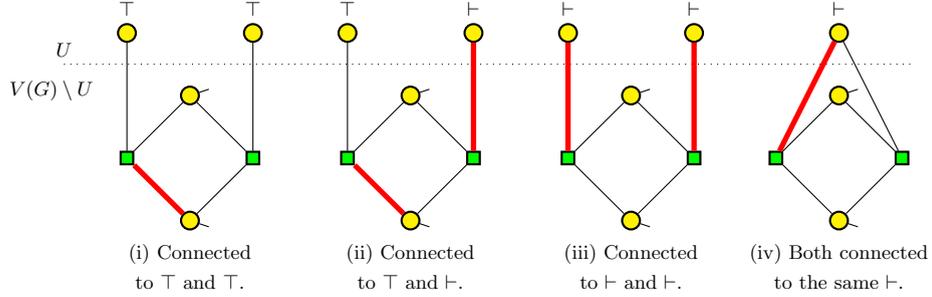
\begin{figure}[t!]
    \centering
    \begin{tikzpicture}[scale=0.83, every node/.style={transform shape}]
      \capt{-2,3.5}{$U$};
      \capt{-2.2,2.8}{$V(G) \sm U$};

      \draw[dotted] (-1,3.5) -- (12.5,3.5);

      \def\x{0}
      \def\y{2}
      \pc{0,0}{}

      \pa{1,1}{a1}
      \pa{1,-1}{a2}
      \pb{2,0}{b1}

      \draw (c) -- (a1);
      \draw (c) -- (a2) [select];
      \draw (b1) -- (a1);
      \draw (b1) -- (a2);

      \paa{0,2}{s1}{label=$\select$}
      \paa{2,2}{s2}{label=$\select$}
      \draw (s1) -- (c);
      \draw (s2) -- (b1);

      \outh{a1}
      \outb{a2}

      \capt{0+\x, -1.8+\y}{(i) Connected };
      \capt{0+\x, -2.2+\y}{to $\select$ and $\select$.};

      \def\x{3.5}
      \def\y{2}
      \pc{0,0}{}

      \pa{1,1}{a1}
      \pa{1,-1}{a2}
      \pb{2,0}{b1}

      \draw (c) -- (a1);
      \draw (c) -- (a2) [select];
      \draw (b1) -- (a1);
      \draw (b1) -- (a2);

      \paa{0,2}{s1}{label=$\select$}
      \paa{2,2}{s2}{label=$\drain$}
      \draw (s1) -- (c);
      \draw (s2) -- (b1)  [select];

      \capt{0+\x, -1.8+\y}{(ii) Connected };
      \capt{0+\x, -2.2+\y}{to $\select$ and $\drain$.};

      \outh{a1}
      \outb{a2}

      \def\x{7}
      \def\y{2}
      \pc{0,0}{}

      \pa{1,1}{a1}
      \pa{1,-1}{a2}
      \pb{2,0}{b1}

      \draw (c) -- (a1);
      \draw (c) -- (a2);
      \draw (b1) -- (a1);
      \draw (b1) -- (a2);

      \paa{0,2}{s1}{label=$\drain$}
      \paa{2,2}{s2}{label=$\drain$}
      \draw (s1) -- (c) [select];
      \draw (s2) -- (b1) [select];

      \outh{a1}
      \outb{a2}

      \capt{0+\x, -1.8+\y}{(iii) Connected };
      \capt{0+\x, -2.2+\y}{to $\drain$ and $\drain$.};

      \def\x{10.3}
      \def\y{2}
      \pc{0,0}{}

      \pa{1,1}{a1}
      \pa{1,-1}{a2}
      \pb{2,0}{b1}

      \draw (c) -- (a1);
      \draw (c) -- (a2);
      \draw (b1) -- (a1);
      \draw (b1) -- (a2);

      \paa{1,2}{s2}{label=$\drain$}
      \draw (s2) -- (c) [select];
      \draw (s2) -- (b1);

      \outh{a1}
      \outb{a2}

      \capt{0+\x, -1.8+\y}{(iv) Both connected };
      \capt{0+\x, -2.2+\y}{to the same $\drain$.};

    \end{tikzpicture}

    \caption{First cases of $C_4^2$. The vertices of $A \cap U$ labeled $\forg$ are treated in the same way than vertices labeled $\drain$.
      If both are connected to the same vertex labeled $\select$,  the case depicted in subfigure (i) applies.
      % \ig{please squeeze this figure to the left}
      %%% There is an \hspace if you think it is better more to the right.
    }
    \label{fig:case3_1}
  \end{figure}

  Secondly, assume that there exists $a_3 \in A  \sm U$ such that $a_3 \in N_G(b_2)$ and $a_3 \not \in N_G(b_1)$.
  Then we are in the case depicted in Fig.~\ref{fig:case3_2}.
  We define $U'$ to be the union of $U$ and the vertices of $G'$, and we define $M'$ and $\gamma'$ by updating the values of $M$ and $\gamma$ according to Fig.~\ref{fig:case3_2}.
  One can check that  properties (\ref{qa}) to (\ref{qg}) are maintained.

\vspace{-.3cm}
  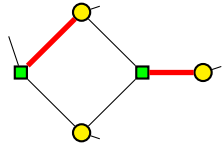
\begin{figure}[H]
    \centering
    \begin{tikzpicture}[scale=0.8, every node/.style={transform shape}]
      \def\x{0}
      \def\y{2}
      \pc{0,0}{}

      \pa{1,1}{a1}
      \pa{1,-1}{a2}
      \pb{2,0}{b1}
      \pa{3,0}{a3}
      \draw (c) -- (a1) [select];
      \draw (c) -- (a2);
      \draw (b1) -- (a1);
      \draw (b1) -- (a2);
      \draw (b1) -- (a3) [select];

      \outh{a1}
      \outb{a2}
      \outh{a3}
      \outtt{c}

    \end{tikzpicture}
    \caption{Second case of $C_4^2$.} \vspace{-.3cm}
    \label{fig:case3_2}
  \end{figure}

  Thirdly, assume that there exist $a_3$ and $a_4$ in $A  \sm U$ such that $a_3 \in N_G(b_2)$, $a_3 \not \in N_G(b_1)$, $a_4 \in N_G(b_1)$, and $a_4 \not \in N_G(b_2)$.
  As in Case~2, this situation is a bit more complicated to handle, but now we do not need to take care of property (\ref{qh}), which simplifies the case analysis compared to Case~2.
  Again, we shall ignore the vertices of $B$ that became of degree $0$ after the removal of the new $U'$ when the situation is favorable to us, i.e., in exactly the same situations as in Case~2, and thus it does not interfere with the fact that the new triple $(U',M', \gamma')$ respects property (\ref{qg}).
  We also sometimes forget the third neighbor of a vertex of $A$ when its existence does not change how to resolve the pattern.
  Using the fact that, by condition~(\ref{c2}), two vertices cannot have the same neighborhood, it follows that the possible patterns are those depicted in Fig.~\ref{fig:case3_3}.
  As Pattern (ii) is not connected to the rest of the graph and has less than $10$ neighbors,  it has already been treated by the algorithm.
  Note that Pattern (iii) cannot exist because of condition~(\ref{c1}).
  For all other patterns, namely Pattern (i) and Patterns (iv) to (ix), we define $U'$ to be the union of $U$ and the vertices of the given pattern, and we define $M'$ and $\gamma'$ by updating the values of $M$ and $\gamma$ according to Fig.~\ref{fig:case3_2}.
  One can check that properties (\ref{qa}) to (\ref{qg}) are maintained.

  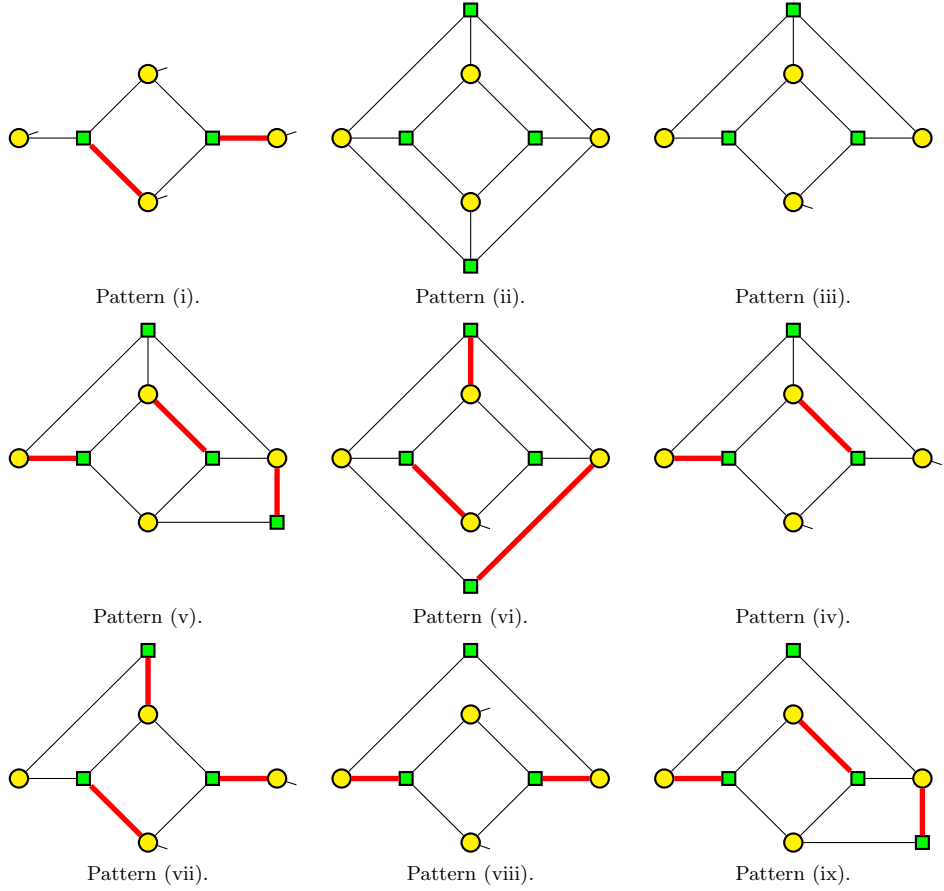
\begin{figure}[H]
    \centering
   \begin{tikzpicture}[scale=0.85, every node/.style={transform shape}]
      \def\x{0}
      \def\y{2}
      \pc{0,0}{}

      \pa{1,1}{a1}
      \pa{1,-1}{a2}
      \pb{2,0}{b1}
      \pa{3,0}{a3}
      \pa{-1,0}{a4}
      \draw (c) -- (a1);
      \draw (c) -- (a2) [select];
      \draw (c) -- (a4);
      \draw (b1) -- (a1);
      \draw (b1) -- (a2);
      \draw (b1) -- (a3)  [select];

      \outh{a1}
      \outh{a2}
      \outh{a4}
      \outh{a3}

      \com{1,-2.5}{Pattern (i).};

      \def\x{5}
      \def\y{2}
      \pc{0,0}{}

      \pa{1,1}{a1}
      \pa{1,-1}{a2}
      \pb{2,0}{b1}
      \pa{3,0}{a3}
      \pa{-1,0}{a4}

      \pb{1,2}{b2}
      \pb{1,-2}{b3}
      \draw (c) -- (a1);
      \draw (c) -- (a2);
      \draw (c) -- (a4);
      \draw (b1) -- (a1);
      \draw (b1) -- (a2);
      \draw (b1) -- (a3);
      \draw (b2) -- (a1);
      \draw (b2) -- (a4);
      \draw (b2) -- (a3);
      \draw (b3) -- (a4);
      \draw (b3) -- (a3) ;
      \draw (b3) -- (a2) ;

      \com{1,-2.5}{Pattern (ii).};

      \def\x{10}
      \def\y{2}
      \pc{0,0}{}

      \pa{1,1}{a1}
      \pa{1,-1}{a2}
      \pb{2,0}{b1}
      \pa{3,0}{a3}
      \pa{-1,0}{a4}

      \pb{1,2}{b2}
      \draw (c) -- (a1);
      \draw (c) -- (a2);
      \draw (c) -- (a4);
      \draw (b1) -- (a1);
      \draw (b1) -- (a2);
      \draw (b1) -- (a3);
      \draw (b2) -- (a1);
      \draw (b2) -- (a4);
      \draw (b2) -- (a3);

      \outb{a2}

      \com{1,-2.5}{Pattern (iii).};

      \def\x{10}
      \def\y{-3}
      \pc{0,0}{}

      \pa{1,1}{a1}
      \pa{1,-1}{a2}
      \pb{2,0}{b1}
      \pa{3,0}{a3}
      \pa{-1,0}{a4}

      \pb{1,2}{b2}
      \draw (c) -- (a1);
      \draw (c) -- (a2);
      \draw (c) -- (a4) [select];
      \draw (b1) -- (a1) [select];
      \draw (b1) -- (a2);
      \draw (b1) -- (a3);
      \draw (b2) -- (a1);
      \draw (b2) -- (a4);
      \draw (b2) -- (a3);

      \outb{a2}
      \outb{a3}

      \com{1,-2.5}{Pattern (iv).};

      \def\x{0}
      \def\y{-3}
      \pc{0,0}{}

      \pa{1,1}{a1}
      \pa{1,-1}{a2}
      \pb{2,0}{b1}
      \pa{3,0}{a3}
      \pa{-1,0}{a4}

      \pb{1,2}{b2}
      \pb{3,-1}{b3}
      \draw (c) -- (a1);
      \draw (c) -- (a2);
      \draw (c) -- (a4) [select];
      \draw (b1) -- (a1) [select];
      \draw (b1) -- (a2);
      \draw (b1) -- (a3);
      \draw (b2) -- (a1);
      \draw (b2) -- (a4);
      \draw (b2) -- (a3);
      \draw (b3) -- (a2);
      \draw (b3) -- (a3) [select];

      \com{1,-2.5}{Pattern (v).};

      \def\x{5}
      \def\y{-3}
      \pc{0,0}{}

      \pa{1,1}{a1}
      \pa{1,-1}{a2}
      \pb{2,0}{b1}
      \pa{3,0}{a3}
      \pa{-1,0}{a4}

      \pb{1,2}{b2}
      \pb{1,-2}{b3}
      \draw (c) -- (a1);
      \draw (c) -- (a2) [select];
      \draw (c) -- (a4);
      \draw (b1) -- (a1);
      \draw (b1) -- (a2);
      \draw (b1) -- (a3);
      \draw (b2) -- (a1) [select];
      \draw (b2) -- (a4);
      \draw (b2) -- (a3);
      \draw (b3) -- (a4);
      \draw (b3) -- (a3) [select];

      \com{1,-2.5}{Pattern (vi).};

      \outb{a2}

      \def\x{0}
      \def\y{-8}
      \pc{0,0}{}

      \pa{1,1}{a1}
      \pa{1,-1}{a2}
      \pb{2,0}{b1}
      \pa{3,0}{a3}
      \pa{-1,0}{a4}

      \pb{1,2}{b2}
      \draw (c) -- (a1);
      \draw (c) -- (a2) [select];
      \draw (c) -- (a4);
      \draw (b1) -- (a1);
      \draw (b1) -- (a2);
      \draw (b1) -- (a3) [select];
      \draw (b2) -- (a1) [select];
      \draw (b2) -- (a4);

      \outb{a2}
      \outb{a3}
      \com{1,-1.5}{Pattern (vii).};

      \def\x{5}
      \def\y{-8}
      \pc{0,0}{}

      \pa{1,1}{a1}
      \pa{1,-1}{a2}
      \pb{2,0}{b1}
      \pa{3,0}{a3}
      \pa{-1,0}{a4}

      \pb{1,2}{b2}
      \draw (c) -- (a1);
      \draw (c) -- (a2);
      \draw (c) -- (a4) [select];
      \draw (b1) -- (a1);
      \draw (b1) -- (a2);
      \draw (b1) -- (a3) [select];
      \draw (b2) -- (a3);
      \draw (b2) -- (a4);

      \outh{a1}
      \outb{a2}
      \com{1,-1.5}{Pattern (viii).};

      \def\x{10}
      \def\y{-8}
      \pc{0,0}{}

      \pa{1,1}{a1}
      \pa{1,-1}{a2}
      \pb{2,0}{b1}
      \pa{3,0}{a3}
      \pa{-1,0}{a4}

      \pb{1,2}{b2}
      \pb{3,-1}{b3}
      \draw (c) -- (a1);
      \draw (c) -- (a2);
      \draw (c) -- (a4) [select];
      \draw (b1) -- (a1) [select];
      \draw (b1) -- (a2);
      \draw (b1) -- (a3);
      % \draw (b2) -- (a1);
      \draw (b2) -- (a4);
      \draw (b2) -- (a3);
      \draw (b3) -- (a2);
      \draw (b3) -- (a3) [select];

      \com{1,-1.5}{Pattern (ix).};

    \end{tikzpicture}

    \caption{Third case of $C_4^2$.}
    \label{fig:case3_3}
  \end{figure}

  \paragraph{{\bf Case 4:}}  there exists $u \in B \sm U$ such that $d_{\bar{U}}(u) = 2$.\smallskip

  \noindent We assume in this case that for every $u \in B \sm U$, $d_{\bar{U}}(u) \geq 2$, and that there is no $C_4$ in $G \gm U$.

  Let $u \in B \sm U$ such that $d_{\bar{U}}(u)=2$.
  Let $\{v_1,v_2\} = N_G(u) \sm U$ and $W = N_G(u) \cap U$, and $U' = U \cup \{u,v_1,v_2\}$.
  Note that both $v_1$ and $v_2$ have at least one neighbor that is not in $U$.
  Note also that $|W| \leq 1$.

  % Assume first that there exists $w \in W$ such that $\gamma(w) = \select$ or for every $w \in W$, $\gamma(w) = \forg$.
  Assume first that $W = \emptyset$ or $W = \{w\}$ and $\gamma(w) \not = \drain$.
  Let $M'$ arise from $M$ by adding the edge $uv_1$.
  Let $\gamma'$ be obtained from $\gamma$ where $\gamma'(v_1) = \select$ and  $\gamma'(v_2) = \drain$.
  % , and for every $w \in W$ such that  $\gamma(w) = \drain$, then $\gamma'(w) = \forg$.
  Clearly, replacing $(U,M,\gamma)$ with $(U',M',\gamma')$, properties (\ref{qa}) to (\ref{qf}) are maintained.
  % Let $n_d$ be the number $|\{w \in  W : \gamma(w) = \drain\}|$.
  % We have that $n_d \leq \Delta -1-k$.
  We obtain that $s' = s+1$, $d' = d + 1$, and $f' = f$.
  These inequalities directly imply that property (\ref{qg}) is maintained.

  Assume now that  $W = \{w\}$ and $\gamma(w) = \drain$. % and for every $w \in W$, $\gamma(w) \not = \select$.
  Let $M'$ arise from $M$ by adding the edge $uw$.
  Let $\gamma'$ be obtained from $\gamma$ where $ \gamma'(v_1) = \gamma'(v_2) = \drain$
  % , and for every $w \in W\sm \{w_1\}$ such that  $\gamma(w) = \drain$, then $\gamma'(w) = \forg$
  and $\gamma'(w) = \select$.
  Clearly, replacing $(U,M,\gamma)$ with $(U',M',\gamma')$, properties (\ref{qa}) to (\ref{qf}) are maintained.
  % Let again $n_d$ be the number $|\{w \in  W : \gamma(w) = \drain\}|$.
  % We have that $n_d \leq \Delta -k$.
  We obtain anew that $s' = s+1$, $d' = d +1$, and $f' = f$.
  These inequalities imply again that property (\ref{qg}) is maintained.

  \paragraph{}
  % Since the considered cases exhaust all possibilities,
  % and in each case we described an extension that maintains the relevant properties,
  % the proof is complete.

  Since the considered cases exhaust all possibilities,
  and in each case we described an extension that maintains the relevant properties,
  the proof is complete up to the running time of the algorithm, which we proceed to analyze.
  One can easily
  % check that each extension  can be checked
  % \ig{what does it mean to ``check an extension?'' You mean to ``check whether an extension operation can be realized''?}
  check whether an extension operation can be realized
  % in polynomial time
  % \jul{As $\Delta$ is fixed, I think we can say that the time is $O(n)$}
  % \ig{Ok, then we should formalize this. I agree that the running time can by bounded by $O(n \Delta^{pattern\ size}$. Just  provide a (short) justification}.
  % in time $O(n \Delta^{p-1})$, where $p =11$ is the size of the biggest pattern.
  in time $O(n)$, where $n= n(G)$.
  Indeed, we consider a constant number of patterns, and
  in each of them we fix a specific vertex.
  Then, for each vertex in $V(G) \setminus U$, we can check whether this vertex corresponds to a specific vertex of one of the patterns
  in constant time, by exploring the neighborhood at distance at most ${p-1}$ from this vertex, where $p =11$ is the size of the largest pattern (cf. Fig.~\ref{fig:case2_4}). As in each extension operation the size of $U$ is incremented by at least one,
  it follows that the overall running time of the algorithm is $O(n^2)$.
  % \jul{Do we leave $\Delta$ inside the running time? I think not.}
  % polynomial \jul{As $\Delta$ is fixed, I think we can say the time is $O(n^2)$} \ig{It is not harmful to specify the precise running time}.
\end{proof}

Equipped with Lemma~\ref{lemmanewDelta}, we are now ready to prove Theorem~\ref{th:approxDelta3}.

%\begin{theorem}
%  \label{th:approxDelta3bis}
%  Let $G$ be a connected bipartite graph of maximum degree at most $3$.
%  Then we can find in polynomial time a uniquely restricted matching $M$ of $G$ of size at least
%  $\frac{5}{9} \nu_{ur}(G)$.
%\end{theorem}

%\ig{I find it strange to use $\Delta$ in the statement, as this applies only to $\Delta = 3$. I know that you did that because of the conjecture. Let's discuss about it}
%
%\noindent\jul{We have a problem that this theorem should be Theorem 2 and not a new number.}
%
%\noindent\jul{We can discuss about the $\Delta = 3$. I leave $\Delta$ as it highlights where the constants come from, but I have not got strong opinion on this and removing it is also ok for me.

\bigskip
\begin{proofof}
  % \ig{In this paragraph, we should say that $\Delta = 3$, right?}
  Again, we give an algorithmic proof such that the running time of the corresponding algorithm is polynomial in $n(G)$.
%  Let $\Delta=3$,
  Let $\alpha=\frac{5}{9}$ and
  let ${\cal G}$ be the set of all bipartite graphs $G$ of maximum degree at most $3$ % such that
  % $G$ does not contains a subgraph $P$ with the shape $R.1$, $R.2$, $R.3$, $R.4$, $R.5$, or $R.6$ depicted in Fig.~\ref{fig:redpat},
  % $G$ does not contains twins vertices,
  such that every component of $G$ has a vertex of degree at most $2$.
  First, we prove that, for every given graph $G$ in ${\cal G}$,
  one can find in polynomial time a uniquely restricted matching $M$ of size at least $\alpha\nu_{ur}(G)$.
  Therefore, let $G$ be in ${\cal G}$.

In order to be able to apply Lemma~\ref{lemmanewDelta}, we apply some reductions. Namely, as long as at least one of the following conditions is fulfilled in $G$, we apply the corresponding reduction, which is described and analyzed below:
  \begin{itemize}
  \item[$\bullet$] {\emph{Condition  (1)}:
      $G$  contains one of the patterns $R.1$, $R.2$, $R.3$, $R.4$, $R.5$, or $R.6$ depicted in Fig.~\ref{fig:redpat}.}
  \item[$\bullet$] {\emph{Condition  (2)}: There exist two vertices in $G$ with the same neighborhood.}
  \item[$\bullet$] {\emph{Condition (3)}: There exists a vertex $u$ in $G$ of degree $1$.}
  \item[$\bullet$] {\emph{Condition (4)}: There exists a vertex $u$ in $G$ of degree $0$.}
  \end{itemize}

  % We give an algorithmic proof of this lower bound whose running time is polynomial in $n(G)$, which yields the claimed result.
  % Therefore,
  % let $G$ be a connected bipartite graph of maximum degree $\Delta$,
  % let $G^*$ be a copy of $G$ (which will be updated by the algorithm), and
  % let \texttt{opt} be a uniquely restricted matching of size $\nu_{ur}(G)$.
  % Our algorithm starts with a preprocessing consisting of four Reductions: Reduction~$\alpha$, Reduction~$\beta$, Reduction~A, and Reduction~B. In order to ease the presentation,  we will use figures to describe the corresponding configurations considered by the algorithm. In all these figures, the partition of $V(G)$ into two sets $A$ and $B$ is represented by using squares and circles, respectively.

  %\ig{We should introduce what the following ``Reductions'' are}
%
%  \ig{also, say something about the fact that the proof is algorithmic}
%
%  \ig{it is too informal to say ``if there is in $G$ a subgraph $P$ with the shape $R.1$''. Everywhere, I would rephrase this type of sentence as ``if $G$ contains a subgraph $P$ isomorphic to the pattern $R.1$ depicted in ... ''}

  \paragraph{Reduction (1).}
  First, if there is in $G$ a subgraph $P$ isomorphic to the pattern $R.1$ or the pattern $R.2$ depicted in Fig.~\ref{fig:redpat}(i) and Fig.~\ref{fig:redpat}(ii), respectively, such that only the vertex $x$ has a neighbor outside of $P$ in $G$,
  then we define $G'$ to be the graph obtained from $G$ by removing every vertex of $P$  except vertex $x$  as depicted in Fig.~\ref{fig:redpat}(i) and Fig.~\ref{fig:redpat}(ii).
  Then, if there is in $G$ a subgraph $P$ isomorphic to the pattern $R.3$ or the pattern $R.4$ depicted in Fig.~\ref{fig:redpat}(iii) and Fig.~\ref{fig:redpat}(iv), respectively, such that only the vertices $x$ and $y$ have a neighbor outside of $P$ in $G$,
  then we define $G'$ to be the graph obtained from $G$ by removing
  every vertex of $P$ except vertices $x$ and $y$ as depicted in Fig.~\ref{fig:redpat}(iii) ad Fig.~\ref{fig:redpat}(iv).
  Finally, if there is in $G$ a subgraph $P$ isomorphic to the pattern $R.5$ or the pattern $R.6$ depicted in Fig.~\ref{fig:redpat}(v) and Fig.~\ref{fig:redpat}(vi), respectively, such that only the vertices $x$, $y$, and $z$ have a neighbor outside of $P$ in $G$,
  then we define $G'$ to be the graph obtained from $G$ by removing
  every vertex of $P$ except vertices $x$, $y$, and $z$ as depicted in Fig.~\ref{fig:redpat}(v) and Fig.~\ref{fig:redpat}(vi).

  Let $M^*$ be the set of red edges depicted in the corresponding figures.
  Then $\nu_{ur}(G')\geq \nu_{ur}(G)-|M^*|$.
  % We claim that in each of the above cases,  \opt\ contains at most the number of red edges depicted in the corresponding figures.
  Indeed, in each case, we can exhaustively check that we cannot select more edges inside the pattern $P$, and in each configuration we provide, in the figures, a solution that leaves vertex $x$ (and vertices $y$ and $z$, if they exist) free to be taken from outside of $P$.
  Moreover, the choice of the red edges is such that they cannot be inside any alternating cycle, whatever the edges we select outside of $P$.
  If $M'$ is a uniquely restricted matching of $G'$,
  then $M'\cup M^*$ is a uniquely restricted matching of $G$.
  %Note that $G'$ belongs to ${\cal G}$.

%  Note that in each case, the new $G$ also belongs to ${\cal G}$.

  \paragraph{Reduction (2).}
  Assume that $v$ and $v'$ are two vertices having exactly the same neighborhood in $G$. We define $G' = G \gm \{v'\}$. %, and then we redefine $G^*$ to be $G'$.
  %Intuitively, we have removed $v'$ from $G$.
  Indeed, $v$ and $v'$ cannot be inside the same uniquely restricted matching, as otherwise there would exist an alternating cycle.
  This implies that $\nu_{ur}(G')\geq \nu_{ur}(G)$.
  Hence, we can safely remove $v'$ from $G$.
  If $M'$ is a uniquely restricted matching of $G'$,
  then $M'$ is a uniquely restricted matching of $G$.
 % Note that $G'$ belongs to ${\cal G}$.

  % \paragraph{Reduction (\ref{c3}).}
  % We arbitrarily choose  a vertex $v$ of $A$,
  % we define $G' = G^* \gm \{v\}$, and then we redefine $G$ to be $G'$.
  % Intuitively, we have removed $v$ from $G^*$.
  % At this point, we may have removed at most one edge that was in \texttt{opt}.

  \paragraph{Reduction (3).}
  If $G$ has a vertex $u$ of degree $1$, and $v$ is the neighbor of $u$, then let $G'=G-\{ u,v\}$.
  Clearly, $\nu_{ur}(G')\geq \nu_{ur}(G)-1$,
  and if $M'$ is a uniquely restricted matching of $G'$,
  then $M'\cup \{ uv\}$ is a uniquely restricted matching of $G$.
  %Note that $G'$ belongs to ${\cal G}$.

  % As long as there is in $G^*$ a vertex $v$ with exactly one neighbor $w$, we define $G' = G^* \gm \{v,w\}$,
  % and then we redefine $G^*$ to be $G'$.
  % Intuitively, we have removed $v$ and $w$ from $G^*$.

  % Note that there is no alternating cycle that uses edge $vw$ because of the degree of $v$.
  % This implies that there exists an optimal solution that contains edge $vw$ and does not contain any other edge with endpoint $w$.
  % The last assertion follows the definition of a matching.
  % We then can safely remove $v$ and $w$ from $G^*$ as a preprocessing, keeping in mind that $vw$ is part of the solution we are building.
  % At this Reduction, we know that when removing $v$ and $w$, we have removed at most one edge of \texttt{opt}, but we have added one edge to our solution.

  \paragraph{Reduction (4).}
  If $G$ has a vertex $u$ of degree $0$, then let $G'=G-\{ u\}$.
  Clearly, $\nu_{ur}(G') = \nu_{ur}(G)$,
  and if $M'$ is a uniquely restricted matching of $G'$,
  then $M'$ is a uniquely restricted matching of $G$.
  %Note that $G'$ belongs to ${\cal G}$.

  \bigskip
  In each of the four reductions defined above, note that the graph $G'$ belongs to ${\cal G}$.
  Similarly to in the proof of Lemma~\ref{lemmanewDelta},
  it can be checked whether a reduction can be realized
  in time $O(n)$, where $n = n(G)$.

  \paragraph{}
  By iteratively repeating these reductions,
  we eventually obtain a set $M_1$ of edges of $G$ as well as a subgraph $G_2$ of $G$ such that
  $G_2\in {\cal G}$,
  $G_2$ does not contain a subgraph $P$ isomorphic to one of the patterns $R.1$, $R.2$, $R.3$, $R.4$, $R.5$, or $R.6$ depicted in Fig.~\ref{fig:redpat},
  $G_2$ does not contain two vertices with the same neighborhood,
  $\nu_{ur}(G_2)\geq \nu_{ur}(G)-|M_1|$,
  $M_1\cup M_2$ is a uniquely restricted matching of $G$ for every uniquely restricted matching $M_2$ of $G_2$, and
  either $n(G_2)=0$ or $\delta(G_2)\geq 2$.
  % \jul{
  % The case where $G_2$ contains an isolated vertex is missing, isn't it?}
  % I think we should explain what happens if $G_2$ contains an isolated vertex.}
  % Note that if $G$ has minimum degree at least $2$, then we may choose $M_1$ empty and $G_2$ equal to $G$.
  Now, by suitably choosing the bipartition of each component $K$ of $G_2$,
  and applying Lemma~\ref{lemmanewDelta} to $K$,
  one can determine in polynomial time a uniquely restricted matching $M_2$ of $G_2$
  with $|M_2|\geq \alpha\nu_{ur}(G_2)$.
  Since the set $M_1\cup M_2$ is a uniquely restricted matching of $G$ of size at least
  $|M_1|+\alpha\nu_{ur}(G_2)\geq |M_1|+\alpha(\nu_{ur}(G)-|M_1|)\geq \alpha\nu_{ur}(G)$,
  the proof of our claim about ${\cal G}$ is complete. Note that the overall running time of the algorithm for graphs in ${\cal G}$ is $O(n^2)$, since each reduction strictly decreases the size of the graph, and the algorithm of Lemma~\ref{lemmanewDelta} also runs in time $O(n^2)$.

  \medskip

  Now, let $G$ be a given connected bipartite graph $G$ of maximum degree at most $3$. If $G$ is not $3$-regular, then $G\in {\cal G}$, and the desired statement already follows. Hence, we may assume that $G$ is $3$-regular, which implies that its two partite sets $A$ and $B$ are of the same order. By \cite{peraso}, we can efficiently decide whether $\nu_{ur}(G)=\nu(G)$. Furthermore, if $\nu_{ur}(G)=\nu(G)$, then, again by \cite{peraso}, we can efficiently determine a maximum matching that is uniquely restricted. Hence, we may assume that $\nu_{ur}(G)<\nu(G)$. This implies that $\nu_{ur}(G)<|A|$, and, hence, there is some vertex $u$ {in $V(G)$ such that}  $\nu_{ur}(G-u)=\nu_{ur}(G)$. Since $G-u\in {\cal G}$ for every vertex $u$ of $G$, considering the $n(G)$ induced subgraphs $G-u$ for $u\in V(G)$, one can determine in polynomial time a uniquely restricted matching $M$ of $G$ with $|M|\geq \max\{\alpha \nu_{ur}(G-u):u\in V(G)\} = \alpha\nu_{ur}(G)$.
  Thus, we obtain an algorithm that finds the desired uniquely restricted matching. Note that since we make $O(n)$ to the algorithm for graphs in ${\cal G}$,  the overall running time is $O(n^3)$.
 % \ig{missing: discuss the running time of the algorithm}
%  \jul{We need to discuss about this. We just say polynomial or do we want the polynomial?} \ig{well, it is better to state the precise running time, as in the previous lemma}
\end{proofof}

\section{Upper bounds on $\chi_{ur}'(G)$}\label{sec:edge-colorings}

Our first result in this section applies to general graphs,
and its proof relies on a natural greedy strategy.
Faudree, Schelp, Gy\'{a}rf\'{a}s, and Tuza conjectured
$\chi_s'(G)\leq \Delta^2$
for a bipartite graph $G$ of maximum degree $\Delta$,
and our Theorem \ref{theorem2} can be considered a weak version of this conjecture.
Theorem \ref{theorem3} below shows that excluding the unique extremal graph
from Theorem \ref{theorem2},
the uniquely restricted chromatic index of bipartite graphs drops considerably.

\begin{theorem}\label{theorem2}
If $G$ is a connected graph of maximum degree at most $\Delta$, then $\chi_{ur}'(G)\leq \Delta^2$ with equality if and only if $G$ is $K_{\Delta,\Delta}$.
\end{theorem}
\begin{proof} Since no two edges of $K_{\Delta,\Delta}$ form a uniquely restricted matching in this graph,
we obtain $\chi_{ur}'(K_{\Delta,\Delta})=|E(K_{\Delta,\Delta})|=\Delta^2$.
Now, let $G$ be a connected graph of maximum degree at most $\Delta$.
We first show that $\chi_{ur}'(G)\leq \Delta^2$.
In a second step, we show that $\chi_{ur}'(G)<\Delta^2$ provided that $G$ is not $K_{\Delta,\Delta}$.

We consider the vertices of $G$ in some linear order, say $u_1,\ldots,u_n$.
For $i$ from $1$ up to $n$,
we assume that the edges of $G$ incident with vertices in $\{ u_1,\ldots,u_{i-1}\}$ have already been colored,
and we color all edges between $u_i$ and $\{ u_{i+1},\ldots,u_n\}$
using distinct colors,
and avoiding any color that has already been used on a previously colored edge incident with some neighbor of $u_i$.
Since $u_i$ has at most $\Delta$ neighbors, each of which is incident with at most $\Delta$ edges,
this procedure requires at most $\Delta^2$ many distinct colors.

Suppose, for a contradiction, that some color class $M$ is not a uniquely restricted matching in $G$.
Since $M$ is a matching by construction, there is an $M$-alternating cycle $C$.
Let $C:u_{r_1}u_{s_1}u_{r_2}u_{s_2}\ldots u_{r_k}u_{s_k}u_{r_1}$
be such that
$r_1$ is the minimum index of any vertex on $C$, and
$u_{r_1}u_{s_k}\in M$.
These choices trivially imply $r_1<s_1$ and $r_1<r_2$.
If $r_2>s_1$, then $u_{r_1}u_{s_k}\in M$ implies that, when coloring the edge $u_{s_1}u_{r_2}$,
some edge incident with the neighbor $u_{r_1}$ of $u_{s_1}$ would already have been assigned the color of the edges in $M$,
and the above procedure would have avoided this color on $u_{s_1}u_{r_2}$.
Therefore, since $u_{r_1}u_{s_k}\in M$ and $u_{r_2}u_{s_1}\in M$, the coloring rules imply $r_2<s_1$,
that is, $r_1<r_2<s_1$.
Now, suppose that $r_i<r_{i+1}<s_i$ for some $i\in [k-1]$.
Since $u_{r_{i+1}}u_{s_i}\in M$ and $u_{r_{i+2}}u_{s_{i+1}}\in M$,
the coloring rules imply in turn
\begin{itemize}
\item[$\bullet$] $r_{i+1}<s_{i+1}$, since otherwise we would have colored $u_{r_{i+1}}u_{s_i}$ differently,
\item[$\bullet$] $r_{i+2}<s_{i+1}$, since otherwise we would have colored $u_{r_{i+2}}u_{s_{i+1}}$ differently, and
\item[$\bullet$] $r_{i+1}<r_{i+2}$, since otherwise we would have colored $u_{r_{i+1}}u_{s_i}$ differently.
\end{itemize}
It follows that $r_{i+1}<r_{i+2}<s_{i+1}$, where we identify $r_{k+1}$ with $r_1$.
Now, by an inductive argument, we obtain $r_1<r_2<\cdots<r_k<r_1$, which is a contradiction.

Altogether, we obtain $\chi_{ur}'(G)\leq \Delta^2$.

\medskip

\noindent Now, let $G$ be distinct from $K_{\Delta,\Delta}$, and we want to prove that
$\chi_{ur}'(G)<\Delta^2$.
Among all uniquely restricted edge colorings of $G$  using colors in $[\Delta^2]$,
we choose a coloring for which the number of edges with color $1$ is as small as possible.
Clearly, we may assume that some edge $uv$ has color $1$, as otherwise we already have that $\chi_{ur}'(G)<\Delta^2$.

If there is a color $\alpha$ in $[\Delta^2]\setminus \{ 1\}$
such that no edge incident with a neighbor of $u$ has color $\alpha$,
then changing the color of $uv$ to $\alpha$
yields a uniquely restricted edge coloring of $G$
with less edges of color $1$, which is a contradiction.
In view of the maximum degree,
this implies that every vertex in $N_G[u]$ has degree $\Delta$,
the set $N_G(u)$ is independent, and,
for every color $\alpha$ in $[\Delta^2]$,
there is exactly one edge incident with a neighbor of $u$ that has color $\alpha$.

%\jul{The notations are usual but I still think that we should define $N_G[u]$ and $N_G(u)$.}

Since $G$ is not $K_{\Delta,\Delta}$,
some neighbor $x$ of $u$ has a neighbor $y$ that does not lie in $N_G(v)$.
Without loss of generality, let $ux$ have color $2$, and let $xy$ have color $3$.
Let $M$ be the set of edges with color $3$.

If $G$ does not contain an $M$-alternating path of odd length at least $3$
between $x$ and a vertex in $N_G(v)\setminus \{ u\}$ that contains the edge $xy$,
then changing the color of $uv$ to $3$
yields a uniquely restricted edge coloring of $G$
with less edges of color $1$, which is a contradiction.
Hence, $G$ contains such a path,
which implies that two edges incident with neighbors of $y$ have color $3$.

If there is a color $\alpha$ in $[\Delta^2]\setminus \{ 1\}$
such that no edge incident with a neighbor of $y$ has color $\alpha$,
then changing the color of $xy$ to $\alpha$ and the color of $uv$ to $3$
yields a uniquely restricted edge coloring of $G$
with less edges of color $1$, which is a contradiction.
Similarly as above, this implies that,
for every color $\alpha$ in $[\Delta^2]\setminus \{ 1,3\}$,
there is exactly one edge incident with a neighbor of $y$ that has color $\alpha$.
Now, changing
the color of $uv$ to $2$,
the color of $ux$ to $3$, and
the color of $xy$ to $2$
yields a uniquely restricted edge coloring of $G$
with less edges of color $1$, which is a contradiction.
This completes the proof.\end{proof}

As observed above, the proof of Theorem \ref{theorem2} is algorithmic;
the simple greedy strategy considered in its first half
efficiently constructs uniquely restricted edge colorings
using at most $\Delta^2$ colors.
Furthermore, also its second half can be turned into an efficient algorithm
that finds uniquely restricted edge colorings using at most $\Delta^2-1$ colors
for connected graphs of maximum degree $\Delta$ that are distinct from $K_{\Delta,\Delta}$; the different cases considered in the proof
correspond to simple manipulations of a given uniquely restricted edge coloring
that iteratively reduce the number of edges of color $1$ down to $0$. Golumbic, Hirst, and Lewenstein~\cite{gohile} showed that deciding whether a given matching is uniquely restricted can be done in polynomial time, and their algorithm can be used to decide which of the simple manipulations can be executed.

Our next goal is to improve Theorem \ref{theorem2} for bipartite graphs.
The following proof was inspired by Lov\'{a}sz's~\cite{lo} elegant proof of Brooks' Theorem.

\begin{lemma}\label{lemma1}
If $G$ is a connected bipartite graph of maximum degree at most $\Delta \geq 4$  that is distinct from $K_{\Delta,\Delta}$,
and $M$ is a matching in $G$, then $M$ can be partitioned into at most $\Delta-1$ uniquely restricted matchings in $G$.
%\jul{It may just be my perception of it, but I do not like this ``at most''. I read it like it is not possible to partition into more than $\Delta$ urm, and this is not true. I prefer ``less than''.}
\end{lemma}
\begin{proof} Let $A$ and $B$ be the partite sets of $G$, and let $R=V(G)\setminus V(M)$.
Note that $M$ is perfect if and only if $R$ is empty.
Whenever we consider a coloring of the edges in $M$, and $\alpha$ is one of the colors,
let $M_{\alpha}$ be the set of edges in $M$ colored with $\alpha$.

First, we assume that $R$ is empty, and that $G$ is not $\Delta$-regular.
By symmetry, we may assume that some vertex $a$ in $A$ has degree less than $\Delta$.
Let $ab\in M$.
Let $T$ be a spanning tree of $G$ that contains the edges in $M$.
Contracting within $T$ the edges from $M$,
rooting the resulting tree in the vertex corresponding to the edge $ab$,
and considering a breadth-first search order,
we obtain the existence of a linear order $a_1b_1,\ldots,a_nb_n$ of the edges in $M$
such that $ab=a_nb_n$, and, for every $i\in [n-1]$, there is an edge between $\{ a_i,b_i\}$ and $\{ a_{i+1},b_{i+1},\ldots,a_n,b_n\}$.
Since $a_n$ has degree less than $\Delta$, this implies that, for every $i\in [n]$,
some vertex $u_i$ in $\{ a_i,b_i\}$ has at most $\Delta-2$ neighbors in $\{ a_1,b_1,\ldots,a_{i-1},b_{i-1}\}$.
Now, we color the edges in $M$ greedily in the above linear order.
Specifically, for every $i$ from $1$ up to $n$, we color the edge $a_ib_i$ with some color $\alpha$ in $[\Delta-1]$ such that,
for every $j\in [i-1]$,
for which $u_i$ has a neighbor in $\{ a_j,b_j\}$,
the edge $a_jb_j$ is not colored with $\alpha$.
By the degree condition on $u_i$, such a coloring exists.
Suppose, for a contradiction, that $M_{\alpha}$ is not uniquely restricted for some color $\alpha$ in $[\Delta-1]$.
Let the edge $a_ib_i$ in $M_{\alpha}$ be such that it belongs to some $M_{\alpha}$-alternating cycle $C$,
and, subject to this condition, the index $i$ is maximum.
If the neighbor of $u_i$ on $C$ outside of $\{ a_i,b_i\}$ is in $\{ a_j,b_j\}$,
then the choice of the edge $a_ib_i$ implies $j<i$,
and the coloring rule implies that the edge $a_jb_j$ is not colored with $\alpha$,
which is a contradiction.
Altogether, the statement follows.

Next, we assume that $R$ is non-empty.
Let $K$ be a component of $G-R$.
Let $M_K$ be the set of edges in $M$ that lie in $K$.
Since $G$ is connected, the graph $K$ is not $\Delta$-regular.
Therefore, proceeding exactly as above, we obtain a coloring of the edges in $M_K$ using the colors in $[\Delta-1]$
such that each color class is a uniquely restricted matching in $K$.
If $K_1,\ldots,K_k$ are the components of $G-R$,
and $M_i$ is a uniquely restricted matching in $K_i$ for every $i\in [k]$,
then $M_1\cup \cdots \cup M_k$ is a uniquely restricted matching in $G$.
Therefore, combining the colorings within the different components,
we obtain that also in this case the statement follows.

At this point, we may assume that $G$ is $\Delta$-regular, and that $M$ is perfect.

Next, we assume that there are two distinct edges $e$ and $e'$ in $M$ such that $V(\{ e,e'\})$
%\jul{I think we should define this notation.}
is a vertex cut of $G$.
This implies that we can partition the set $M\setminus \{ e,e'\}$ into two non-empty sets $M_1$ and $M_2$
such that there is no edge between $V(M_1)$ and $V(M_2)$.
For $i\in [2]$, let $G_i$ be the subgraph of $G$ induced by $V(\{ e,e'\}\cup M_i)$.
Since $G$ is connected, the graph $G_i$ is not $\Delta$-regular.
In view of the above,
this implies that there is a coloring $c_i$ of the edges of the perfect matching $\{ e,e'\}\cup M_i$ of $G_i$
using the colors in $[\Delta-1]$ such that each color class of $c_i$ is a uniquely restricted matching in $G_i$.
If $c_i(e)\not=c_i(e')$ for both $i$ in $[2]$, then we may assume that $c_1$ and $c_2$ assign the same colors to $e$ and $e'$,
and it is easy to verify that the common extension $c$ of $c_1$ and $c_2$ to $M$
has the property that every color class of $c$ is a uniquely restricted matching in $G$.
Hence, we may assume that necessarily $c_1(e)=c_1(e')$.
Note that this implies in particular that at least one of the two possible edges between $V(\{ e\})$ and $V(\{ e'\})$ is missing.

Let $c_1(e)=\alpha$.
Let $e=ab$, $e'=a'b'$, and $U=\{ a,b,a',b'\}$.
For every vertex $u\in U$,
let $C_1(u)$ be the set of colors $\beta$ for which $M_1$ contains an edge $vw$ with $c_1(vw)=\beta$ such that $u$ is adjacent to $v$ or $w$.
If there is some $u\in U$ and some color $\beta\in ([\Delta-1]\setminus \{ \alpha\})\setminus C_1(u)$,
then changing the color of the unique edge in $\{ e,e'\}$ incident with $u$ from $\alpha$ to $\beta$
yields a coloring $c_1'$ of the edges in $\{ e,e'\}\cup M_1$
using the colors in $[\Delta-1]$ such that each color class of $c_1'$ is a uniquely restricted matching in $G_1$.
Furthermore, $c_1'(e)\not=c_1'(e')$, which is a contradiction.
This implies that $[\Delta-1]\setminus \{ \alpha\}\subseteq C_1(u)$ for every $u\in U$.
In particular, each vertex $u$ in $U$ has at least $\Delta-2$ neighbors in $V(M_1)$,
and, hence, at most one neighbor in $V(M_2)$.
Let $C_2(u)$ for $u\in U$ be defined analogously as above.
Clearly, the set $C_2(a)\cup C_2(a')$ contains at most two distinct colors.
Since $\Delta-1\geq 3$, we may assume that $c_2$ is such that the set $C_2(a)\cup C_2(a')$ does not contain the color $\alpha$.
Now, let $c_2'$ be a coloring of the edges in $\{ e,e'\}\cup M_2$
that coincides with $c_2$ on $M_2$ and colors $e$ and $e'$ with color $\alpha$.
It is easy to see that each color class of $c_2'$ is a uniquely restricted matching in $G_2$.
Let $c$ be the common extension of $c_1$ and $c_2'$ to $M$.
Suppose, for a contradiction,
that the color class $M_{\beta}$ of $c$ is not uniquely restricted for some color $\beta$ in $[\Delta-1]$.
Clearly, we have $\beta=\alpha$.
Let $C$ be an $M_{\alpha}$-alternating cycle in $G$.
It is easy to see that $C$ contains both edges $e$ and $e'$
%, but no edge between $\{ a,b\}$ and $\{ a',b'\}$. \igb{why $C$ contains no edge between $\{ a,b\}$ and $\{ a',b'\}$? I don't agree with this. I think that we are not using the right argument here. I propose the following: }
Furthermore, since at least one of the two possible edges between $\{a,b\}$ and $\{a',b'\}$ is missing, it follows that $C$ contains an edge between $\{ a,a'\}$ and $V(M_2)$.
%\igb{then the following sentence should be deleted}
%Furthermore, it follows that $C$ contains an edge between $\{ a,a'\}$ and $V(M_2)$.
Since $c$ coincides with $c_2$ on $M_2$, and $C_2(a)\cup C_2(a')$ does not contain $\alpha$,
we obtain a contradiction.

Altogether, we may assume that there are no two distinct edges $e$ and $e'$ in $M$ such that $V(\{ e,e'\})$ is a vertex cut of $G$.

Now, we show the existence of three edges $ab$, $a'b'$, and $a''b''$ in $M$ such that
some of the two possible edges between $\{ a',b'\}$ and $\{ a'',b''\}$ is missing,
and
either $a$ is adjacent to $b'$ as well as $b''$
or $b$ is adjacent to $a'$ as well as $a''$.
Therefore, let $a_1b_1$ be an edge in $M$.
Let $a_2b_2,\ldots,a_{\Delta}b_{\Delta}$ be the edges in $M$ such that $N_G(a_1)=\{ b_1,\ldots,b_{\Delta}\}$.
We may assume that $\{ a_2,b_2,\ldots,a_{\Delta},b_{\Delta}\}$ induces a complete bipartite graph $K_{\Delta-1,\Delta-1}$;
otherwise, we find the three edges with the desired properties.
Since $G$ is not $K_{\Delta,\Delta}$, the vertex $b_1$ is non-adjacent to some vertex $a_i$ in $\{ a_2,\ldots,a_{\Delta}\}$.
Now, if $a_j\in \{ a_2,\ldots,a_{\Delta}\}\setminus \{ a_i\}$,
then one of the two possible edges between $\{ a_1,b_1\}$ and $\{ a_i,b_i\}$ is missing,
and $b_j$ is adjacent to $a_1$ as well as $a_i$.
Altogether, we obtain three edges $ab$, $a'b'$, and $a''b''$ in $M$ with the desired properties.

By symmetry, we may assume that
$a$ is adjacent to $b'$ and $b''$,
and $a'$ is non-adjacent to $b''$.
In view of the above,
the graph $G'=G-V(\{ a'b',a''b''\})$ is connected, and $M'=M\setminus \{ a'b',a''b''\}$ is a perfect matching of $G'$.
Let $T'$ be a spanning tree of $G'$ that contains the edges in $M'$.
Contracting within $T'$ the edges from $M'$,
rooting the resulting tree in the vertex corresponding to the edge $ab$,
and considering a breadth-first search order,
we obtain the existence of a linear order $a_3b_3,\ldots,a_nb_n$ of the edges in $M'$
such that $ab=a_nb_n$, and, for every $i\in [n-1]\setminus [2]$,
there is an edge between $\{ a_i,b_i\}$ and $\{ a_{i+1},b_{i+1},\ldots,a_n,b_n\}$.
Now, we color the edges in $M$ greedily in the linear order $a_1b_1,a_2b_2,a_3b_3,\ldots,a_nb_n$,
where $a_1b_1=a''b''$ and $a_2b_2=a'b'$.
Note that, for every $i\in [n-1]\setminus [2]$,
some vertex $u_i$ in $\{ a_i,b_i\}$ has at most $\Delta-2$ neighbors in $\{ a_1,b_1,\ldots,a_{i-1},b_{i-1}\}$.
We color $a_1b_1$ and $a_2b_2$ with the same color.
For every $i$ from $3$ up to $n-1$,
we color the edge $a_ib_i$ with a color $\alpha$ in $[\Delta-1]$ such that,
for every $j\in [i-1]$,
for which $u_i$ has a neighbor in $\{ a_j,b_j\}$,
the edge $a_jb_j$ is not colored with $\alpha$.
By the degree condition on $u_i$, such a coloring exists.
Finally, since $a_n$ has neighbors in the two edges $a_1b_1$ and $a_2b_2$ that are colored with the same color,
there is some color $\alpha$ in $[\Delta-1]$ for which no edge $a_ib_i$ with $i\in [n-1]$ such that $a_n$ is adjacent to $b_i$,
is colored with $\alpha$,
and we color the edge $a_nb_n$ with that color $\alpha$.
Suppose, for a contradiction, that $M_{\beta}$ is not uniquely restricted for some color $\beta$ in $[\Delta-1]$.
Let the edge $a_ib_i$ in $M_{\beta}$ be such that it belongs to some $M_{\beta}$-alternating cycle $C$,
and, subject to this condition, the index $i$ is maximum.
Since $a'$ is non-adjacent to $b''$, we have $i\geq 3$.
Let $u_n=a_n$.
If the neighbor of $u_i$ on $C$ outside of $\{ a_i,b_i\}$ is in $\{ a_j,b_j\}$,
then the choice of the edge $a_ib_i$ implies $j<i$,
and the coloring rule implies that the edge $a_jb_j$ is not colored with $\beta$,
which is a contradiction. This completes the proof.
\end{proof}

Lemma \ref{lemma1} fails for $\Delta=3$;
the matching $\{ a_1b_1,a_2b_2,a_3b_3,a_4b_4,a_5b_5\}$ of the graph $G$ in Fig. \ref{fig1}
cannot be partitioned into two uniquely restricted matchings.
Note that the matching $\{ a_1b_3,a_2b_1,a_3b_5,a_4b_2,a_5b_4\}$ though
is the union of the two uniquely restricted matchings
$\{ a_1b_3,a_3b_5\}$ and $\{ a_2b_1,a_4b_2,a_5b_4\}$.

\begin{figure}[htb]
\begin{center}
%TeXCAD Picture [1.pic]. Options:
%\grade{\on}
%\emlines{\off}
%\epic{\off}
%\beziermacro{\on}
%\reduce{\on}
%\snapping{\on}
%\pvinsert{% Your \input, \def, etc. here}
%\quality{8.000}
%\graddiff{0.005}
%\snapasp{1}
%\zoom{13.4543}
\unitlength 1.3mm % = 2.845pt
\linethickness{0.4pt}
\ifx\plotpoint\undefined\newsavebox{\plotpoint}\fi % GNUPLOT compatibility
\begin{picture}(41,19)(0,0)
\put(0,5){\circle*{1}}
\put(10,5){\circle*{1}}
\put(20,5){\circle*{1}}
\put(30,5){\circle*{1}}
\put(40,5){\circle*{1}}
\put(0,1){\makebox(0,0)[cc]{$a_1$}}
\put(10,1){\makebox(0,0)[cc]{$a_2$}}
\put(20,1){\makebox(0,0)[cc]{$a_3$}}
\put(30,1){\makebox(0,0)[cc]{$a_4$}}
\put(40,1){\makebox(0,0)[cc]{$a_5$}}
\put(0,15){\circle*{1}}
\put(10,15){\circle*{1}}
\put(20,15){\circle*{1}}
\put(30,15){\circle*{1}}
\put(40,15){\circle*{1}}
\put(0,19){\makebox(0,0)[cc]{$b_1$}}
\put(10,19){\makebox(0,0)[cc]{$b_2$}}
\put(20,19){\makebox(0,0)[cc]{$b_3$}}
\put(30,19){\makebox(0,0)[cc]{$b_4$}}
\put(40,19){\makebox(0,0)[cc]{$b_5$}}
\put(0,15){\line(0,-1){10}}
\put(0,5){\line(1,1){10}}
\put(10,15){\line(0,-1){10}}
\put(20,15){\line(-2,-1){20}}
\put(20,5){\line(0,1){10}}
\put(20,5){\line(-2,1){20}}
\put(0,15){\line(1,-1){10}}
\put(30,15){\line(1,-1){10}}
\put(40,5){\line(0,1){10}}
\put(40,15){\line(-1,-1){10}}
\put(30,5){\line(0,1){10}}
\put(20,15){\line(2,-1){20}}
\put(40,15){\line(-2,-1){20}}
\put(30,15){\line(-2,-1){20}}
\put(10,15){\line(2,-1){20}}
\end{picture}
\end{center}
\vspace{-.3cm}
\caption{A bipartite graph $G$.}\label{fig1}
\end{figure}
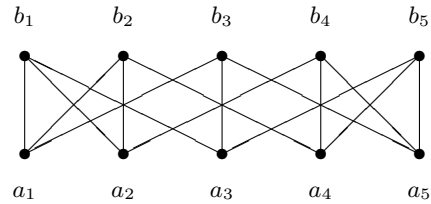
\noindent Lemma \ref{lemma1} also fails for non-bipartite graphs;
in fact, if $G$ arises from the disjoint union of two copies of $K_{\Delta}$ by adding a perfect matching $M$,
then every partition of $M$ into uniquely restricted matchings requires $\Delta$ sets.

With Lemma \ref{lemma1} at hand, the proof of our final result is easy.

\begin{theorem}\label{theorem3}
If $G$ is a connected bipartite graph of maximum degree at most $\Delta \geq 4$ that is distinct from $K_{\Delta,\Delta}$,
then $\chi'_{ur}(G)\leq \Delta^2-\Delta$.
\end{theorem}
\begin{proof}
Since $G$ is bipartite, its edge set can be partitioned into $\Delta$ matchings~\cite{lopl}.
By Lemma~\ref{lemma1}, each of these matchings can be partitioned into $\Delta-1$ uniquely restricted matchings. This completes the proof.
\end{proof}

%\medskip

\noindent Note that the graph $G$ in Fig.~\ref{fig1}
also satisfies $\chi'_{ur}(G)\leq \Delta^2-\Delta=9-3=6$.
In fact,
the uniquely restricted matchings
$\{ a_1b_1,a_4b_2,a_5b_4\}$,
$\{ a_1b_2,a_2b_4,a_5b_5\}$,
$\{ a_2b_1,a_3b_3,a_4b_5\}$,
$\{ a_1b_3,a_4b_4\}$,
$\{ a_2b_2,a_3b_5\}$, and
$\{ a_3b_1,a_5b_3\}$
partition $E(G)$.

\section{Concluding remarks}\label{sec:conclusions}

Our results motivate several open problems.
As stated above, we believe that the conclusion of Theorem~\ref{theoremnew1}
holds without the assumption of $C_4$-freeness.
We also believe that better approximation factors are possible,
and that approximation lower bounds in terms of the maximum degree could be proved.
One could study the approximability of
the uniquely restricted matching number
in other classes of graphs.
Finally,
complexity results concerning the uniquely restricted chromatic index
should be provided.

\bibliographystyle{abbrv}
\bibliography{urm}

%\newpage
%\begin{appendix}
%\section{Proof of Theorem~\ref{th:approxDelta3}}
%\label{ap:Delta3}
%\input{Delta3.tex}
%\end{appendix}

\end{document}